\documentclass{amsart}

\usepackage{amsmath}
\usepackage{amssymb}
\usepackage{graphicx}
\usepackage{color}
\usepackage[all]{xy}
\usepackage{amsthm}

\usepackage{enumerate}
\usepackage{tikz}
\usetikzlibrary{shapes.geometric}
\usepackage{overpic}
\usepackage{bbding}
\usepackage{enumerate}

\usepackage[normalem]{ulem}

\usepackage[colorlinks,citecolor=blue,linkcolor=blue]{hyperref}
\usepackage[nameinlink]{cleveref}

\theoremstyle{plain}
\newtheorem{thm}{Theorem}[section]

\newtheorem{lem}[thm]{Lemma}
\crefname{lem}{Lemma}{Lemmas}

\newtheorem{prop}[thm]{Proposition}
\crefname{prop}{Proposition}{Propositions}

\newtheorem{claim}[thm]{Claim}
\crefname{claim}{Claim}{Claims}

\newtheorem{notation}[thm]{Notation}
\crefname{notation}{Notation}{Notations}

\newtheorem*{namedtheorem}{\theoremname}
\newcommand{\theoremname}{testing}
\newenvironment{named}[1]{\renewcommand{\theoremname}{#1}\begin{namedtheorem}}{\end{namedtheorem}}

\theoremstyle{definition}

\newtheorem*{dfn}{Definition}

\newtheorem{rem}[thm]{Remark}

\numberwithin{equation}{section}

\newcommand{\bdy}{\partial}

\newcommand{\from}{\colon} 

\renewcommand{\setminus}{\smallsetminus}

\newcommand{\C}{\mathbb{C}}
\renewcommand{\H}{\mathbb{H}}

\newcommand{\Q}{\mathbb{Q}}
\newcommand{\R}{\mathbb{R}}
\newcommand{\Z}{\mathbb{Z}}
\newcommand{\N}{\mathbb{N}}

\newcommand{\Fp}{\mathbb{F}_p}
\newcommand{\pp}{{\mathfrak{p}}}
\newcommand{\ok}{{\mathcal{O}_k}}
\newcommand{\okp}{{\mathcal{O}_{k_{\pp}}}}

\newcommand{\PSL}{\mathrm{PSL}}
\newcommand{\SL}{\mathrm{SL}}
\renewcommand{\ss}{\mathbf{s}}
\renewcommand{\tt}{\mathbf{t}}

\newcommand{\Hth}{\mathbb{H}^3}

\newcommand{\smod}[1]{{\!\!\pmod{#1}}}
\newcommand{\trace}{\operatorname{tr}}

\def\Fp{\mathbb{F}_p}

\def\pp{{\mathfrak{p}}}

\def\ok{{\mathcal{O}_{\rm{k}}}}
\def\okp{{\mathcal{O}_{{\rm{k}}_{\pp}}}}
\def\oKp{{\mathcal{O}_{{\rm{K}}_{\pp}}}}

\def\oK{{\mathcal{O}_{\rm{K}}}}

\def\traceSym{\tau}
\def\traceSymTwo{\tau}
\newcommand{\nclose}[1]{\ensuremath{\langle\!\langle#1\rangle\!\rangle}}

\newcommand{\PrimeSet}{\Omega}

\usepackage{accents}
\newlength{\dhatheight}

\definecolor{bettergreen}{rgb}{0,0.6,0.4}
\definecolor{mutedgreen}{rgb}{.1,.75,0.15}

\definecolor{purple}{rgb}{0.4,0,0.6}

\usepackage{caption}
\usepackage{subcaption}
\usepackage{subcaption}
\title{Conjugacy Distinguished Cosets in Hyperbolic $3$-Manifold Groups}

\author{David Futer}
 \address{Department of Mathematics, Temple University, Philadelphia, PA 19122}
\email{dfuter@temple.edu}

\author{Emily Hamilton}
 \address{Department of Mathematics,
California Polytechnic State University,
San Luis Obispo, CA 93407}
\email{mhamil09@calpoly.edu}

\author{Neil R. Hoffman}
 \address{Department of Mathematics and Statistics, University of Minnesota, Duluth, MN 55812}
\email{neilhoff@d.umn.edu}

\begin{document}

\begin{abstract}
A subset $S$ of a group $G$ is  \emph{conjugacy distinguished} if the union of all conjugates of $S$ is closed in the profinite topology on $G$.
We prove that if $M = \H^3/\Gamma$ is a hyperbolic $3$-manifold of finite volume, $g \in \Gamma$,  and $H$ is an abelian subgroup of $\Gamma$,
then the coset $gH$ is conjugacy distinguished in $\Gamma$.  A subset $S \subset G$ is \emph{conjugacy distinguished from a class of subgroups}  if, 
for every $K$ in the class  that  is disjoint from the union of conjugates of $S$, 
there exists a homomorphism $\varphi \from G \rightarrow F$, where $F$ is a finite group, such that $\varphi(K)$ is disjoint from the union of conjugates of $\varphi(S)$.
In previous work, we proved that if $M = \H^3/\Gamma$ is a hyperbolic $3$-manifold of finite volume,
then a coset of a maximal parabolic subgroup with cusp $C$ is conjugacy distinguished from  the class of maximal parabolic subgroups of $\Gamma$ with cusps distinct from $C$.
We extend this result by proving that a coset of a 
loxodromic subgroup is conjugacy distinguished from the class of maximal parabolic subgroups of $\Gamma$.
\end{abstract}

\maketitle

\section{Introduction}
\label{introduction}

In 1911, Max Dehn posed three foundational problems in combinatorial group theory: the word problem, the conjugacy problem,
and the isomorphism problem.  Each problem is undecidable in general.  However, algorithms have been
developed for many classes of groups.   These algorithms have been linked to the notions of subgroup separability and conjugacy separability.

\begin{dfn}
Let $G$ be a group. The \emph{profinite topology} on $G$ is the topology whose basic open sets are cosets of finite-index normal subgroups. 
Since every coset of a finite-index subgroup $H \lhd G$ is the complement of finitely many other cosets of $H$, the basic open sets are also closed.
\begin{enumerate}[\:\:$(1)$]
\item A subset $S \subset G$ is called \emph{separable} if it is closed in the profinite topology on $G$.
Equivalently, for every element $g \in G \setminus S$, there is a homomorphism $\varphi \from G \to F$, where $F$ is a finite group, 
such that $\varphi(g) \notin \varphi(S)$.
\item  A group $G$ is \emph{residually finite} if the trivial subgroup is separable.
\item A group $G$ is \emph{subgroup separable} if every finitely generated subgroup of $G$ is separable.
\item A group $G$ is \emph{conjugacy separable} if every conjugacy class in $G$ is separable. 
Equivalently, for every pair of elements $g, h \in G$ that are not conjugate in $G$, there is a homomorphism
$\varphi \from G \to F$, where $F$ is a finite group, such that $\varphi(g)$ and $\varphi(h)$ are not conjugate.
\end{enumerate}
\end{dfn}

Let $\Gamma$ be a finitely presented group.  If  $\Gamma$ is residually finite, then $\Gamma$ 
has a solvable word problem. In broad strokes, an algorithm to solve the problem uses two parallel processes: given a word $w$ in the generators of $\Gamma$, one process tries to express $w$ as a product of (conjugates of) the relators, while a second process tries to find a homomorphism from $\Gamma$ to a finite group where the image of $w$ is nontrivial.  If $w$ is trivial, the first process will always terminate. If $w$ is non-trivial, residual finiteness ensures the second process will terminate. Thus, for any word, 
one of the two processes must terminate, solving the problem.
 If $\Gamma$ is subgroup separable, then a similar two-process algorithm solves the subgroup membership problem for $\Gamma$.  Finally, if $\Gamma$ is conjugacy separable, then a similar algorithm solves the conjugacy problem for $\Gamma$.  {Under these schemes, there is a particular satisfaction in that each positive answer can be affirmed via a computation in the group and each negative answer (non-triviality, non-membership, a lack of conjugacy) comes in tandem with a homomorphism to a finite group, which serves as independently verifiable evidence of the negative answer.

Hall proved that free groups are subgroup separable \cite{Hall}; see Stallings~\cite{Stallings:FiniteGraphs} for a particularly beautiful proof. Scott proved that surface groups are subgroup separable, by embedding them into right-angled reflection groups \cite{Scott}. The same ideas extend to the fundamental groups of  compact Seifert fibered spaces. Agol \cite{Agol:VirtualHaken} and Wise \cite{Wise:QCH-Monograph} proved that the fundamental groups of hyperbolic $3$--manifolds also virtually embed into right-angled Coxeter groups, thereby establishing subgroup separability. On the other hand, there exist $3$--manifold groups that are not subgroup separable; see \cite{BKS} and \cite{LongNiblo} for examples.

All of the above classes of groups are also conjugacy separable. For free groups, this follows from the work of Stebe \cite{Stebe}. For surface groups and hyperbolic $3$--manifold groups, this follows from the combined work of Agol, Haglund, Kahn, Markovic, Minasyan, and Wise \cite{Agol:VirtualHaken, HaglundWise:Coxeter, KahnMarkovic:SurfaceSubgroups, Minasyan:Hereditary, Wise:QCH-Monograph}.  
Indeed, such a group embeds as a quasi-convex subgroup of a right-angled Artin group, and is furthermore a virtual retract of the right-angled Artin group. Since right-angled Artin groups are hereditarily conjugacy separable by a theorem of Minasyan, it follows  that virtual retracts of right-angled Artin groups are hereditarily conjugacy separable. See \cite[implication (H.8)]{AFW} for a more detailed explanation. Finally, Hamilton, Wilton, and Zalesskii extended these ideas to prove that all fundamental groups of compact, orientable $3$--manifolds are conjugacy separable \cite{HWZ:Separability}. 
  
We are primarily concerned with concepts related to conjugacy separability.

\begin{dfn} Let $G$ be a group.
\begin{enumerate}[\:\:$(1)$]
\item A subset $S$ of $G$ is said to be \emph{conjugacy distinguished} if $\cup_{g \in G}  \ gSg^{-1}$ is separable. Equivalently,
for every element $\gamma \in G$ that is not conjugate to an element of $S$, there is a homomorphism $\varphi \from G \to F$, where $F$
is a finite group, such that $\varphi(\gamma)$ is not conjugate to an element of $\varphi(S)$.
\item We say that $G$ is \emph{subgroup conjugacy distinguished} if every finitely generated subgroup of $G$ is
conjugacy distinguished. 
\end{enumerate}
\end{dfn}

Virtually free groups \cite{Dyer:SeparatingConjugates} and limit groups \cite{CZ3} are known to be subgroup conjugacy distinguished. Chagas and Zalesskii proved 
that this property also holds for $3$-manifold groups \cite{CZ2}.

\subsection{Main results}
In this paper, we look at cosets of abelian subgroups of hyperbolic $3$-manifold groups. In particular, we prove the following.

\begin{thm}
\label{Theorem:ConjugacyDistinguished}
Let $M = \H^3 / \Gamma$ be a hyperbolic $3$-manifold of finite volume, let $g \in \Gamma$, and
let $H$ be an abelian subgroup of $\Gamma$. Then the coset $gH$ is conjugacy distinguished.
\end{thm}

Here is a topological interpretation of \Cref{Theorem:ConjugacyDistinguished}, in the special case where  $H$ is a maximal parabolic subgroup of $\Gamma$ with fixed point $p \in \bdy \H^3$, and $g \notin H$. Consider the cusp $C$ of $M$ corresponding to $p$, as well as the lifts of $C$ in covers of $M$.
The statement that  $gH$ is conjugacy distinguished in $\Gamma$ means the following: for every $\gamma \in \Gamma$ that is not conjugate into $gH$, there exists a finite-sheeted regular covering $\widehat M = \H^3/\widehat{\Gamma}$
 such that no conjugate of $\gamma$  in $\Gamma$ maps the cusp in $\widehat M$ corresponding to $p$ to the cusp corresponding to $g(p)$. Equivalently, $\gamma$ does not map the $k(p)$ cusp in $\widehat M$ 
to the $k(g(p))$ cusp for any $k \in \Gamma$. The subgroup $\widehat{\Gamma}$ can be taken to be the kernel of a homomorphism from $\Gamma$ to a finite group in which no conjugate of the image of $\gamma$ lies in the image of the coset $gH$.

If $S \subset \Gamma$ is conjugacy distinguished, then we can separate elements from all conjugates of $S$. 
It is also useful to separate subgroups from all conjugates of $S$. This leads to the following.

\begin{dfn}
A subset $S \subset G$ is  \emph{conjugacy distinguished from a class of subgroups} $\mathcal{C}$ of $G$ if, 
for every $K \in \mathcal{C}$ that  is disjoint from the union of conjugates of $S$, 
there exists a homomorphism $\varphi \from G \rightarrow F$, where $F$ is a finite group, such that $\varphi(K)$ is disjoint from the union of conjugates of $\varphi(S)$.
\end{dfn}

In previous work \cite{FHH}, we proved that if $M = \H^3/\Gamma$ is a hyperbolic $3$-manifold of finite volume,
then a coset of a maximal parabolic subgroup with cusp $C$ is conjugacy distinguished from the class of maximal parabolic subgroups of $\Gamma$ with cusps distinct from $C$. 
This separability result was a crucial ingredient in proving a geometric result: that every cusped hyperbolic $3$--manifold $M$ has a finite cover admitting infinitely many geometric ideal triangulations.

In this paper, we extend the separability result from \cite{FHH} to cosets of loxodromic subgroups.

\begin{thm}
\label{Theorem:Separability}
Let $M = {\H}^3 / \Gamma$ be a hyperbolic $3$-manifold of finite volume. 
Let $H$ be a loxodromic subgroup of $\Gamma$ and let $K$ be a maximal parabolic subgroup of $\Gamma$. 
Let $g \in \Gamma$ be an element such that $K$ is disjoint from every conjugate of $gH$. 
Then there exists a homomorphism $\varphi\from \Gamma \rightarrow G$, where $G$ is a finite group, such that
$\varphi(K)$ is disjoint from every conjugate of $\varphi(gH)$.
\end{thm}

\Cref{Theorem:Separability} has the following topological interpretation,
in the case where $H$ is a maximal loxodromic subgroup. Let $\widetilde \alpha \subset \Hth$ be the geodesic in $\H^3$ stabilized by $H$. 
 Given $g \in \Gamma \setminus H$, the coset $gH$ is the set of all elements of $\Gamma$ that move $\widetilde \alpha$ to  another geodesic, namely $g \widetilde \alpha$. Let $\widetilde \beta$ be an arc that connects $\widetilde \alpha$ to $g \widetilde \alpha$. Projecting everything down to $M$, we see a closed geodesic $\alpha$ and an arc $\beta$ from $\alpha$ to $\alpha$.
Now, if $\widehat M$ is the finite cover of $M$ corresponding to the subgroup
  $\widehat \Gamma = \varphi^{-1} \circ \varphi(K)$, then the cusp corresponding to $K$ lifts to $\widehat M$
  and \emph{every} preimage of $\beta$ connects distinct preimages of $\alpha$.

We note that there are many situations where a cusp $C$ corresponding to a parabolic subgroup $K$ \emph{never} lifts to a regular cover of $M$. For instance, if $M$ is a knot complement in $S^3$, then $\Gamma \cong \pi_1(M)$ is normally generated by a meridian loop $\mu$ belonging to a parabolic subgroup $K$. So the only normal subgroup containing $K$ is $\Gamma$ itself. In such a situation, if one wants every preimage of $\beta \subset M$ to have some property in a cover where $C$ lifts, one is necessarily forced to work with irregular covers. Despite this difficulty, \Cref{Theorem:Separability} still provides an irregular cover with the desired properties.

\subsection{Proof strategy}\label{Sec:Strategy}
To prove \Cref{Theorem:ConjugacyDistinguished} and \Cref{Theorem:Separability}, we must construct appropriate finite quotients of $\Gamma$. By standard results, the discrete torsion-free group $\Gamma \subset PSL(2,\C)$ admits a lift to $SL(2, \C)$, and after conjugation we may assume that the entries of  $\Gamma$ lie in a finitely generated ring $R$ contained in a number field $k$.
If $S$ is a finite quotient of $R$, then we obtain a reduction homomorphism $$\varphi: \Gamma \hookrightarrow \SL(2, R) \rightarrow \SL(2, S).$$
Thus, the problem reduces to constructing suitable finite quotients of the ring $R$. 
To accomplish this, we use tools from algebraic number theory.
In \Cref{algebra}, we review the necessary background and prove a new subgroup separability result 
for the multiplicative group of units in a finitely generated ring. The new result, \Cref{Prop:Algebraic}, can be summarized as follows.
Let $\lambda \in R$,
and let $\{ \omega_i \} \subset R$ be a finite
collection of nonzero elements such that each multiplicative subgroup $\langle \lambda, \omega_i \rangle \subset k^{*}$ is free abelian of rank two.
We show that there exist a finite field $F$ and a ring homomorphism $\eta\from R \rightarrow F$ such that 
$\eta(\omega_i) \notin \langle \eta(\lambda) \rangle$ for all $i$.
Moreover, we can prescribe the multiplicative order of $\eta(\lambda)$ in $F^{\ast}$. 
This provides a controlled finite quotient of $R$ in which the given cyclic subgroup remains separated from finitely many fixed elements.
Together with a previously established result regarding additive subgroups of $R$ (\Cref{Prop:AddRingSep}), this
is the main algebraic tool for producing the finite quotients of $\Gamma$ required in the proofs of our main theorems. 

In some cases, when it does not suffice to work with the original representation of $\Gamma$ into $\SL(2,k)$, we turn to Dehn filling. We choose a long tuple of slopes $\ss$ along some cusps of $M$, fill those cusps to obtain a hyperbolic manifold $M(\ss)$, and then find reduction homomorphisms of $\pi_1(M(\ss))$ using the above algebraic techniques. The quotient homomorphism from $\Gamma \cong \pi_1(M)$ to   $\pi_1(M(\ss))$ needs to have the property that 
 certain loxodromic elements of $\pi_1(M)$ remain loxodromic, and that certain non-conjugate elements remain non-conjugate. While this is hard to prove algebraically, it turns out to be manageable using negatively curved Dehn filling. 
 
In a hyperbolic manifold $M$, conjugacy of elements in $\pi_1(M)$ is equivalent to free homotopy of loops, and every loxodromic loop is freely homotopic to a unique closed geodesic in the hyperbolic metric. The same properties hold in any negatively curved metric. We replace certain cusps of $M$ by solid tori of variable negative curvature, using the Gromov--Thurston $2\pi$--Theorem \cite{Bleiler-Hodgson:TwoPi}, while maintaining exactly the same metric on some compact core $M_0 \subset M$. If a closed geodesic of $M$ is contained in this compact core $M_0$, then it will still be a geodesic in the negatively curved metric on $M(\ss)$, and if two distinct closed geodesics are contained in $M_0$, then the corresponding loops will still belong to distinct free homotopy classes in $M(\ss)$. While this idea is easy to carry out for a finite set of loops in $M$, we are able to employ the method for all the closed geodesics representing all the loxodromics in an entire coset $gH$. We carry out these arguments in
\Cref{geometry}.

In \Cref{main}, we combine the above ingredients to prove \Cref{Theorem:ConjugacyDistinguished} and \Cref{Theorem:Separability}. The proofs involve detailed case analyses, using the traces of elements to obstruct conjugacy. A key step is \Cref{Prop:trace lemma}, which can be paraphrased as follows: if $H$ is a loxodromic subgroup and there are no elements in $gH$ with the same trace as $\gamma \in \Gamma$, then there is a finite quotient of $\Gamma$ where this property is preserved. This proposition is needed in the proofs of both main results.

\subsection*{Acknowledgements} The authors acknowledge the use of Claude to proofread the paper for typos and notational consistency. No AI was used for any other purpose in this paper. Futer was supported in part by NSF grant DMS--2405046.

\section{Coarse geometry, geodesic axes, and negatively curved Dehn filling}
\label{geometry}

In this section, we employ coarse geometric arguments in negatively curved metrics on a cusped $3$--manifold $M$ and its Dehn fillings. As previewed in \Cref{Sec:Strategy}, our goal is to prove that certain loxodromic elements of $\pi_1(M)$ remain loxodromic, and certain non-conjugate elements remain non-conjugate. We begin with arguments in coarse geometry based on the work of Cannon \cite{Cannon:CocompactHyperbolic} and then proceed to Dehn filling.

\subsection{Coarse geometry}\label{Sec:Coarse}
Given an interval $I \subset \R$ and constants $K, K' > 0$, a path $f \from I \to \H^n$ is called a \emph{$(K, K')$--quasigeodesic} if the following double inequality holds for every $x, y \in I$:
\[
\frac{1}{K} d(x,y) - K' \leq d(f(x), f(y)) \leq K \cdot  d(x,y) + K'. 
\]

We will need the following facts about quasigeodesics in $\H^n$. While both statements hold in the generality of $\delta$--hyperbolic spaces, we only need the results for $\H^n$, as presented in Cannon's prescient paper \cite{Cannon:CocompactHyperbolic}.

\begin{thm}[Local-to-global; see {\cite[Lemma 2 and Figure 7]{Cannon:CocompactHyperbolic}}]\label{Thm:LocalToGlobal}
For every angle $\theta > 0$, there is a constant $C = C(\theta)$ and 
 constants $K,K'>0$ such that the following holds. If $P$ is a piecewise geodesic path in $\H^n$ consisting of segments longer than $C$ that meet at angle greater than $\theta$, then $P$ is a $(K,K')$--quasigeodesic.
\end{thm}

For instance, if $\theta > \pi/3$, one may compute that $C = \log 3$ suffices.

\begin{thm}[Morse lemma; {\cite[Theorem 2 and Lemma 1]{Cannon:CocompactHyperbolic}}]\label{Thm:Morse}
Given $K, K' >0$, there is a constant $W>0$ such that the following holds. If $P$ and $Q$ are  $(K,K')$ quasigeodesics from $x \in \overline{\H^n}$ to $y \in \overline{\H^n}$, then $P$ is contained in the $W$--neighborhood of $Q$, and vice versa. In particular, the hyperbolic geodesic from $x$ to $y$  is contained in the $W$--neighborhood of $P$.
\end{thm}

We can now apply Theorems~\ref{Thm:LocalToGlobal} and~\ref{Thm:Morse} to establish a fact about the closed geodesics representing elements of a coset.

\begin{lem}
\label{Lem:AxesInNeighborhood}
Let $M = \H^3 / \Gamma$ be a hyperbolic $3$--manifold, let $g\in \Gamma$, and let $H = \langle h \rangle$ be a loxodromic subgroup of $\Gamma$. Let $\alpha$ be the closed geodesic in the free homotopy class of $h$.
Then there is a constant $W > 0$, depending on $g$ and $H$, such that every loxodromic element of $gH$ corresponds to a closed geodesic contained in the $W$--neighborhood of $\alpha$.
\end{lem}

\begin{proof}
We may assume without loss of generality that $H$ is a maximal abelian subgroup. (If not, then replace $H$ by the maximal abelian subgroup that contains it). Thus $\alpha$ is a primitive geodesic. If $g \in H$, then every nontrivial element of $gH = H$ is a loxodromic whose closed geodesic is some power of $\alpha$. Thus, if $g \in H$, the conclusion of the lemma holds for an arbitrary $W > 0$.

We now assume that $g \notin H$. 
Choose a basepoint $x$ on $\alpha$, and let $\beta$ be the shortest loop based at $x$ representing the element $g$ in  $\pi_1(M, x)$. Thus $\beta$ is the projection to $M$ of a geodesic segment  $\widetilde \beta \subset \H^3$, which begins at $\widetilde x \in \widetilde \alpha$ and ends at $g (\widetilde x) \in g (\widetilde \alpha)$.

Choose a large radius $r > 0$, so that the neighborhood $N = N_r(\widetilde \beta)$ has the following properties:
\begin{itemize}
\item At each intersection point of  $(\widetilde \alpha \cup g\widetilde \alpha) \cap \bdy N$, the angle of intersection is greater than $\pi/3$.

\item The four points of  $(\widetilde \alpha \cup g\widetilde \alpha) \cap \bdy N$ are located at pairwise distance greater than $C = C(\pi/3)$.
\end{itemize}

The constants in the two bulleted conditions are chosen with an eye toward applying \Cref{Thm:LocalToGlobal}.
Consider a piecewise geodesic path $\gamma$  that follows a long segment of $\widetilde \alpha$ (whose portion outside $N$ is longer than $C$), then follows $\widetilde \beta$, then another long segment of $g\widetilde \alpha$. Now, form a path $\gamma'$ by replacing  $\gamma \cap N$ by the geodesic segment between the two points of $\gamma \cap \bdy N$.
(See \Cref{Fig:Quasigeodesic}.) Then $\gamma'$ is a piecewise geodesic path whose segments are longer than $C$ and whose angles are greater than $\pi/3$, which is a $(K, K')$ quasigeodesic by \Cref{Thm:LocalToGlobal}.

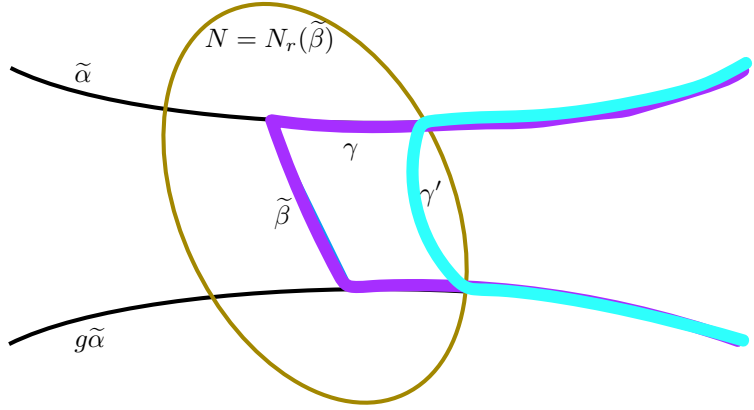
\begin{figure}
\begin{overpic}[width=4in]{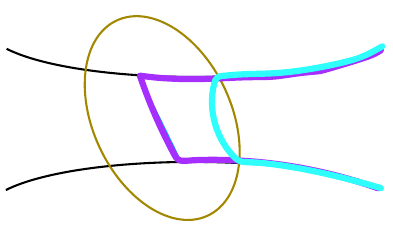}
\put(10,43.5){$\widetilde \alpha$}
\put(10,9){$g \widetilde \alpha$}
\put(36,25){$\widetilde \beta$}
\put(27,48){$N = N_r(\widetilde \beta)$}
\put(45,34){$\gamma$}
\put(55,28){$\gamma'$}
\end{overpic}
\vspace{-1ex}
\caption{A piecewise geodesic path $\gamma$  can be modified inside the neighborhood $N = N_r(\widetilde \beta)$ to a path $\gamma'$ that must be a quasigeodesic by \Cref{Thm:LocalToGlobal}.}
\label{Fig:Quasigeodesic}
\end{figure}

Consider a loxodromic element $g h^n \in gH$. Then $gh^n$ stabilizes a bi-infinite piecewise geodesic path  $P_n$ in $\H^3$ of the following form: starting at  $\widetilde x$, follow the lift $\widetilde \beta$ to $g \widetilde x$,  then follow $|n|$ fundamental domains of $\alpha$ contained in a single geodesic line covering $\alpha$, then another lift of $\beta$, then another $|n|$ fundamental domains of $\alpha$, and so on. When $|n|$ is sufficiently large, any portion of $P_n$ between consecutive lifts of $\beta$ will be longer than $(C+2r)$. Now, we modify $P_n$ to a path $P'_n$ as follows. Every time $P_n$ traverses some translate $\sigma(\widetilde \beta)$ for $\sigma \in \Gamma$, we take a shortcut through the corresponding translate $\sigma(N_r(\widetilde \beta))$, exactly as in  \Cref{Fig:Quasigeodesic}. Since $P'_n$ consists of 
segments longer than $C$ joined by angles greater than $\pi/3$, \Cref{Thm:LocalToGlobal} implies that $P'_n$ is a $(K, K')$ quasigeodesic for some uniform $(K, K')$.

The element $g h^n$ also stabilizes a geodesic line $\widetilde \gamma_n$, whose endpoints on $\bdy_\infty \H^3$ agree with those of $P_n$ and $P'_n$.
By the Morse Lemma (\Cref{Thm:Morse}), there is a uniform constant $W > 0$ such that $\widetilde \gamma_n$ lies within a $W$--neighborhood of $P'_n$. By increasing $W$ to incorporate the diameter of $N = N_r(\widetilde \beta)$, we also ensure that every point of $\widetilde \gamma_n$ lies within a $W$--neighborhood of some translate of $\widetilde \alpha$. The projection of $\widetilde \gamma_n$ to $M$ is a closed geodesic lying in a $W$--neighborhood of $\alpha$.

This proves the conclusion of the lemma for elements $g h^n \in gH$ where $|n|$ is sufficiently large. Since there are only finitely many values of $n$ failing this hypothesis, simply taking the worst-case scenario proves the lemma for all loxodromics in $gH$.
\end{proof}

\subsection{Dehn filling}
We will apply \Cref{Lem:AxesInNeighborhood} in the context of Dehn filling. We recall the setup. Let $M = \H^3/\Gamma$ be a hyperbolic $3$-manifold of finite volume with distinct cusps $C_1, \dots, C_\ell$ (which may or may not be all the cusps of $M$). Let $B_1, \ldots, B_\ell$ be disjoint horospherical neighborhoods of  $C_1, \ldots, C_\ell$.  For every cusp torus $\bdy B_i$,  we choose a slope $s_i$, and package these slopes in a tuple $\ss = (s_1, \ldots, s_\ell)$. The Dehn filled manifold $M(\ss)$ is constructed from $M$ by excising each $B_i$ and replacing it with a solid torus whose meridian disk is glued to $s_i$. On the level of group theory, we have $\pi_1(M(\ss)) \cong \pi_1(M) / \nclose{s_1, \ldots, s_\ell}$, and there is a quotient homomorphism $\psi_\ss \from \pi_1(M) \to \pi_1(M(\ss))$.

When each $s_i$ is sufficiently long (as measured in the Euclidean metric on $\bdy B_i$), Thurston's Dehn surgery theorem says that $M(\ss)$ admits a complete hyperbolic metric. However, for this section, we will use a different choice of metric, with variable negative curvature.

Suppose that the Euclidean representative of $s_i$ on the torus $\bdy B_i$ is strictly longer than $2\pi$. Then the Gromov--Thurston $2\pi$ theorem \cite[Theorem 9]{Bleiler-Hodgson:TwoPi} guarantees that the Dehn filled manifold $M(\ss)$ admits a metric of variable negative curvature, obtained by replacing each cusp neighborhood $B_i$ by a negatively curved solid torus $V_i$. In particular,  $M \setminus (\bigcup B_i)$ embeds isometrically into $M(\ss)$ with this metric.

The negatively curved metric on $M(\ss)$ is locally $CAT(-\kappa)$ for some $\kappa < 0$. (See \cite[Theorem 2.1]{FKP:Volume} for explicit estimates on curvature in terms of the Euclidean lengths of slopes $s_1, \ldots, s_\ell$.) Accordingly, the elements of $\pi_1(M(\ss))$ can be identified as \emph{elliptic}, \emph{parabolic}, and \emph{hyperbolic} as follows:
\begin{itemize}
\item The only elliptic element of $\pi_1 (M(\ss))$ is the identity.
\item An element $\sigma \in \pi_1 (M(\ss))$ is parabolic if and only if a loop representing $\sigma$ is freely homotopic into a cusp neighborhood (corresponding to one of the cusps of $M$ that did not get filled).
\item An element $\sigma \in \pi_1 (M(\ss))$ is hyperbolic if and only if a loop representing $\sigma$ is freely homotopic to a closed geodesic. This geodesic is unique, because $M(\ss)$ is negatively curved.
\end{itemize}
We emphasize that for this section, the quotient homomorphism $\psi_\ss \from \pi_1(M) \to \pi_1(M(\ss))$ is considered on the level of abstract groups only. By the above bullets, the identification of an element $\psi_\ss(\gamma)$ as elliptic/parabolic/hyperbolic is insensitive to the exact choice of negatively curved metric. 

In \Cref{main} of this paper, we will need to consider the complete hyperbolic metric on $M(\ss)$ and work with the algebraic properties of entries in the matrices corresponding to  elements of $\pi_1(M(\ss))$. However, for this section, we will only work with the metric of variable negative curvature constructed using the $2\pi$--theorem. The main advantage of this viewpoint is that the metrics on $M$ and $M(\ss)$ agree on the nose outside the cusp neighborhoods and filled solid tori. 

\begin{lem}
\label{Lem:RemainLoxodromic}
Let $M = \H^3/\Gamma$ be a hyperbolic 3-manifold of finite volume, with at least one cusp. Let $X$ be any set of loxodromic elements in $\Gamma$, with the property that the union of all closed geodesics corresponding to the elements of $X$ is contained in a compact subset $R  \subset M$. Then, for any sufficiently long Dehn  filling tuple $\ss = (s_1, \ldots, s_\ell)$,  the image of every element of $X$ remains loxodromic in $\pi_1(M(\ss))$. Furthermore, if $\gamma_1, \gamma_2 \in X$ are not conjugate in $\pi_1(M)$, then $\psi_\ss(\gamma_1)$ and $\psi_\ss(\gamma_2)$ are not conjugate in $\pi_1(M(\ss))$.
\end{lem}

\begin{proof}
Let $B_1, \ldots, B_\ell$ be disjoint horospherical neighborhoods of cusps  $C_1, \ldots, C_\ell$, which are small enough to avoid the compact set $R$. Then, let $\ss = (s_1, \ldots, s_\ell)$ be any tuple of slopes on  $C_1, \ldots, C_\ell$, such that the Euclidean length of $s_i$ on $\bdy B_i$ is longer than $2\pi$. As mentioned above, the Gromov--Thurston $2\pi$ theorem guarantees that $M(\ss)$ admits a negatively curved metric that is obtained from the metric on $M$ by replacing each $B_i$ with a negatively curved solid torus.

Now, let $\gamma \in X \subset \pi_1(M)$. By hypothesis, a loop in $M$ representing $\gamma$ is freely homotopic to a closed geodesic contained in $R$. Since the metric on $R$ is unchanged in $M(\ss)$, the same closed curve is a closed geodesic in $M(\ss)$ representing the free homotopy class of $\psi_\ss(\gamma)$. Thus $\psi_\ss(\gamma)$ is still loxodromic in $\pi_1(M(\ss))$. 

Finally, suppose that $\gamma_1, \gamma_2 \in X$ are not conjugate. Then they are represented by 
 distinct closed, oriented geodesics contained in $R$. (These geodesics are not necessarily primitive, and different powers of a curve are considered distinct.) Since $R$ embeds isometrically into $M(\ss)$, these closed geodesics are still distinct in $M(\ss)$, and the lemma follows.
\end{proof}

\begin{lem}
 \label[lemma]{Conjugate abelian implies equal}
If $M = \H^3/\Gamma$ is a hyperbolic $3$-manifold of finite volume and $K$ is a maximal abelian subgroup of $\Gamma$, then two elements 
of $K$ are conjugate in $\Gamma$ if and only if they are equal.
\end{lem}

\begin{proof}
Let $x, y \in K$ and suppose that $x = z y z^{-1}$ for some $z \in \Gamma$. Then $z$ stabilizes the fixed point(s) of $K$ at infinity.
Since $K$ is a maximal abelian subgroup, it follows that $z \in K$. Therefore, $x = y$, as required.
\end{proof}

\begin{lem}
\label{Lem:NotConjugateSimple}
Let $M = \H^3/\Gamma$ be a hyperbolic $3$-manifold of finite volume, with at least one cusp. 
Let $X \subset \Gamma$ be a finite set whose elements 
are pairwise non-conjugate in $\Gamma$.  Then, for every sufficiently  
long Dehn filling $\ss$, the map $\psi_{\ss}$ is injective on $X$, and
the elements of $\psi_{\ss}(X)$  are pairwise  non-conjugate in $\pi_1(M(\ss))$.
\end{lem}

\begin{proof} 
Let $X = \{\sigma_1, ..., \sigma_n \}$. We may assume without loss of generality that if $\sigma_i, \sigma_j$ are conjugate into 
the same abelian subgroup, then they already belong to the same abelian subgroup.
We may also assume without loss of generality that $1 \in X$.

Let $R \subset M$ be the union of all closed geodesics corresponding to loxodromic elements of $X$. Then $R$ is compact, because $X$ is finite.

Next, we choose a tuple of filling slopes $\ss$ satisfying the following criteria:
\begin{enumerate}[\:\:$(1)$]
\item\label{Itm:LongR} The tuple $\ss$ is sufficiently long to satisfy \Cref{Lem:RemainLoxodromic} for the given compact set $R$. In particular, $R$ embeds isometrically into $M(\ss)$ with its negatively curved metric.
\item\label{Itm:LongSigma} If $\sigma_i$ and $\sigma_j$ both belong to a common maximal parabolic subgroup $K$, then the Euclidean length of any filling slope on the cusp corresponding to $K$ is strictly longer than the length of $\sigma_i  \sigma_j^{-1}$. In particular, since $1 \in X$, the Euclidean length of the filling slope is longer than the length of every $\sigma_i \in K$, which ensures that $\psi_\ss(\sigma_i) \neq 1$.
\end{enumerate}
We emphasize that both criteria will be satisfied when $\ss$ is sufficiently long.

{Now, assume that $\ss$ is a tuple of slopes on \emph{all} the cusps of $M$ satisfying criteria \eqref{Itm:LongR} and \eqref{Itm:LongSigma}. Then $M(\ss)$ is a closed $3$--manifold whose fundamental group contains no parabolics. We claim that $\psi_\ss \from \pi_1(M) \to \pi_1(M(\ss))$ satisfies the conclusion of the lemma.} 

{The claim can be checked as follows.} First, by \Cref{Lem:RemainLoxodromic}, if $\sigma_i$ and $\sigma_j$ are loxodromic elements of $X$, then $\psi_\ss(\sigma_i)$ and $\psi_\ss(\sigma_j)$ remain loxodromic and non-conjugate 
in $\pi_1(M(\ss))$. Second, if one of $\sigma_i$ and $\sigma_j$ is parabolic and the other is loxodromic, then the loxodromic element will remain loxodromic (with axis contained in $R$) while the parabolic element will be mapped into a loxodromic subgroup carried by a negatively curved solid torus disjoint from $R$. In particular, $\psi_\ss(\sigma_i)$ and $\psi_\ss(\sigma_j)$ are not conjugate. Third, if $\sigma_i$ and $\sigma_j$ are parabolic elements corresponding to distinct cusps of $M$, then they are mapped to nontrivial elements in loxodromic subgroups carried by disjoint negatively curved solid tori, which means that $\psi_\ss(\sigma_i)$ and $\psi_\ss(\sigma_j)$ are not conjugate.

Fourth, suppose that $\sigma_i$ and $\sigma_j$ are both parabolic elements corresponding to the same cusp, corresponding to a maximal parabolic subgroup $K$. Then $\psi_\ss(\sigma_i)$ and $\psi_\ss(\sigma_j)$ will both belong to the same abelian subgroup $\psi_\ss(K)$. By \Cref{Conjugate abelian implies equal}, $\psi_\ss(\sigma_i)$ and $\psi_\ss(\sigma_j)$ can only be conjugate if they are equal. But the Dehn filling slope on the cusp corresponding to $K$ is strictly longer than $\sigma_i  \sigma_j^{-1}$, which implies $\psi_\ss(\sigma_i  \sigma_j^{-1}) \neq 1$. {Thus $\psi_\ss$ satisfies all the conclusions of the lemma.}

{Finally, suppose that $\tt$ is a tuple of slopes on only some cusps of $M$, satisfying \eqref{Itm:LongR} and \eqref{Itm:LongSigma}. We complete $\tt$ to a tuple of slopes $\ss$ on all the cusps, still satisfying \eqref{Itm:LongR} and \eqref{Itm:LongSigma}. By the above argument, the filling homomorphism $\psi_\ss \from \pi_1(M) \to \pi_1(M(\ss))$ satisfies the conclusion of the lemma. But by construction, $\psi_\ss$ factors through $\psi_\tt \from \pi_1(M) \to \pi_1(M(\tt))$, so $\psi_\tt$ must also satisfy the desired conclusions.}
\end{proof}

\begin{lem}
\label{Lem:ExcludedSubgroupSimple}
Let $M = \H^3/\Gamma$ be a hyperbolic $3$-manifold of finite volume. Let $K$ be a maximal 
parabolic subgroup of $\Gamma$.
Let $H$ be a subgroup of $K$, and suppose $k_0 \in K \setminus H$. Then there exist arbitrarily long Dehn filling tuples $\ss$, such that $\psi_{\ss}(k_0) \notin \psi_{\ss}(H)$.
\end{lem}

\begin{proof}
Let  $C_1, \ldots, C_\ell$ be the cusps of $M$, so that $C_1 = A$ corresponds to the maximal parabolic subgroup $K$. Fix a horospherical neighborhood $B_1$ about $C_1$. We will make a careful choice of slope $s_1$ on $\bdy B_1$, and then complete $s_1$ to a long tuple of slopes on all the cusps.

Since the image of $k_0$ in $K/H$ is nontrivial, the residual finiteness of $K/H$ provides a finite cyclic group $G$ and a homomorphism $\varphi \from K \to G$ that factors through $K/H$, with the property that $k_0 \notin \ker(\varphi)$. Since $G \cong \Z/n\Z$ for some $n > 1$, we may choose a generating set $a,b$ for $K \cong \Z^2$ so that $\ker(\varphi)$ is generated by $a$ and $b^n$. By construction, $H \subset \ker(\varphi)$.

Now, observe that for every $m \in \Z$, the element $q_m := a  b^{mn} \in \ker(\varphi)$ is primitive in $K$, because $K$ is generated by $b$ and $a b^{mn}$. Thus $q_m$ is represented by a simple closed curve on $\bdy B_1$, i.e.\ a Dehn filling slope. As $m \to \infty$, the Euclidean length of  the slope defined by $q_m$ also grows to $\infty$. 
Furthermore, the coset $k_0 \langle q_m \rangle \subset k_0 \ker(\varphi)$ is disjoint from  $\ker(\varphi)$, hence disjoint from $H$. Thus, in any Dehn filling quotient $\psi_\ss \from \pi_1(M) \to \pi_1(M(\ss))$ whose effect on $K$ is to quotient by $\langle q_m \rangle$, we have $\psi_\ss(k_0) \notin \psi_\ss(H)$.

We build a tuple $\ss$ as follows. Let $s_1$ be the slope on $\bdy B_1$ represented by $q_m = a b^{mn}$ for some $m$, and let $s_2, \ldots, s_\ell$ be arbitrary long slopes on $C_2, \ldots, C_\ell$. By choosing $m \gg 0$ and choosing $s_2, \ldots, s_\ell$ sufficiently long, we obtain an arbitrarily long tuple of slopes $\ss$. Since $q_m \in \ker(\varphi)$, the above argument ensures $\psi_\ss(k_0) \notin \psi_\ss(H)$.
\end{proof}

Combining the above results, we obtain:

\begin{prop}
\label[prop]{Prop:not conjugate}
Let $M = \H^3/\Gamma$ be a hyperbolic $3$-manifold of finite volume. Let $K$ be a maximal 
parabolic subgroup of $\Gamma$.
Let $H$ be subgroup of $K$, and suppose $k_0 \in K \setminus H$.
Fix an element $\gamma \in \Gamma$, and let $X \subset \Gamma$ be a finite set 
such that no element of $X$ is conjugate to $\gamma$ in $\Gamma$.  Then there exist generators
$h_1$ and $h_2$ of $K$, representing slopes that can be completed to tuples $\ss_1$ and 
$\ss_2$, such that the following hold:
\begin{enumerate}[\:\:$(1)$]
\item\label{Itm:HyperbolicFilling} For $j = 1,2$, the filled manifold $M(\ss_j)$ is hyperbolic.
\item\label{Itm:NotConjugate} For $j = 1,2$, no element in $\psi_{\ss_j}(X)$  is conjugate to $\psi_{\ss_j}(\gamma)$ in $\Gamma(\ss_j)$.
\item\label{Itm:propThree} For $j = 1$ only, $\psi_{\ss_1}(k_0) \notin \psi_{\ss_1}(H)$.
\end{enumerate}
\end{prop}

\begin{proof}
{Let $A$ be the cusp of $M$ corresponding to $K$.}
By \Cref{Lem:ExcludedSubgroupSimple}, we may find an arbitrarily long Dehn filling tuple $\ss_1$, 
whose coordinate on cusp $A$ is represented by an element $h_1 \in K$, such that $\psi_{\ss_1}(k_0) \notin \psi_{\ss_1}(H)$. Thus, conclusion \eqref{Itm:propThree} holds. Choosing $\ss_1$ sufficiently long ensures that $M(\ss_1)$ is hyperbolic by
Thurston's hyperbolic Dehn surgery theorem, so conclusion \eqref{Itm:HyperbolicFilling} holds for 
$\ss_1$. 
By \Cref{Lem:NotConjugateSimple} applied to $X \cup \{ \gamma \}$, choosing $\ss_1$ sufficiently long also ensures that conclusion \eqref{Itm:NotConjugate} holds for $\ss_1$. 

Now, $h_1 \in K$ is a very long primitive element. There are infinitely many ways to complete $h_1$ to a generating set $h_1, h_2$ of $K$, and the slopes represented by these choices of $h_2$ become arbitrarily long on the cusp $A$. Thus, by choosing a sufficiently long slope $h_2$ and then completing $h_2$ to a tuple $\ss_2$, we ensure that conclusions  \eqref{Itm:HyperbolicFilling}  and \eqref{Itm:NotConjugate} also hold for $\ss_2$.
\end{proof}

\section{Algebraic preliminaries}
\label{algebra}
In this section and the next, we will work extensively with the number-theoretic properties of traces of matrices representing hyperbolic isometries.
When considering quotients $\Hth/\Gamma$, we will assume that $\Gamma$ is a subgroup of $\PSL(2,\C)$ acting on $\Hth$ by isometries. We will always lift $\Gamma$ to $\SL(2,\C)$, using the following theorem.

\begin{thm}[Thurston~\cite{Thurston:Notes}, Bass~\cite{bass1980groups}, Culler--Shalen~\cite{CuSh}]
\label{Thm:ThurstonLiftConjugate}
Let $M = \Hth/\Gamma$ be a hyperbolic $3$-manifold of finite volume. Then 
\begin{enumerate}[\:\:$(1)$]
\item\label{Itm:RepLifts} $\Gamma \subset \PSL(2,\C)$ can be lifted to $\SL(2,\C)$. 
\item\label{Itm:ConjugateRing} $\Gamma$ can be conjugated in $\SL(2,\C)$ to lie in 
$\SL(2,R) \subset \SL(2,k)$, where $R$ is a finitely generated subring of a number field $k$.
\end{enumerate}
\end{thm}

Recall that a \emph{number field} is a finite field extension of $\mathbb{Q}$. 
We  now describe some needed algebraic results.
The first of these is  \cite[Theorem 2.6]{HWZ:Separability}, which allows us to distinguish an element  $g \in R$ from a finitely generated submodule of $R$ that does not contain $g$.

\begin{prop}[Hamilton--Wilton--Zalesskii \cite{HWZ:Separability}]
\label[prop]{Prop:AddRingSep}
Let $R$ be a finitely generated ring in a number field $k$.
By fixing a $\Q$ embedding of $k$ into $\C$, we may view $k \subset \C$.
Let $\omega \in R$, and set
$Z_\omega = \{ m + n\omega \mid m, n \in \mathbb{Z} \}$ and $Q_\omega = \{ m + n\omega \mid m, n \in \Q \}$.
If $y \in R \setminus Q_{\omega}$, then there exist a finite ring $S$
and a ring homomorphism $\rho \from R \rightarrow S$ such that
$\rho(y) \notin \rho(Z_{\omega})$. \qed
\end{prop}

In previous work \cite[Lemma 4.8]{FHH}, the authors proved a technical variant of \Cref{Prop:AddRingSep}. This variant
allows us to establish a number of mod $p$ congruences that are key to contradictions needed in the proofs of the main theorems of this paper.

\begin{lem}[{\cite[Lemma 4.8]{FHH}}]
\label[lem]{Lem:LinearIndep}
Let $R$ be a finitely generated ring in a number field $k$.  By fixing a $\Q$ embedding of $k$ into $\C$,
we may view $k \subset \C$.  Let $\omega \in R \setminus \R$. Then there is an infinite collection  $\PrimeSet$ of primes of $\Q$, such that for each prime $p \in \PrimeSet$, there exist a finite field $F_{\pp}$ 
of characteristic $p$ and a ring homomorphism $\eta_p\from R \rightarrow F_{\pp}$, such that $\{ 1, \eta_p(\omega) \}$ is linearly independent over $\Fp$. 

Consequently, $\eta_p$ has the following property. Consider an element $y_*  \in Q_\omega =\{ m + n\omega \mid m, n \in \Q \}$. Express $y_*$ in lowest terms:
\[
y_* = \frac{m_* + n_* \omega}{v_*} \quad 
\]
where $m_*, n_* \in \mathbb{Z}$ and $v_* \in \mathbb{N}$. 
If $\eta_p(y_*) = \eta_p(m + n \omega)$ for some $m, n \in \mathbb{Z}$, then $v_* m \equiv m_* \smod p$ and $v_* n \equiv n_* \smod p$.
\end{lem}

The next two propositions help control the multiplicative order of elements in finite quotients. 

\begin{prop}[{\cite[Corollary 2.5]{H2}}]
\label[prop]{Prop:orderm}
Let $R$ be a finitely generated ring
in a number field $k$, let $\lambda$ be
a non-zero element of $R$
that is not a root of unity, and let
$x_1, x_2, \ \ldots \ , x_j$ be non-zero
elements of $R$.
Then there exists a positive integer $n$
with the following property.  For each integer
$m \geq n$, there exist a finite field $F$
and a ring homomorphism $\eta \from R \rightarrow F$
such that the multiplicative order of $\eta(\lambda)$ is
equal to $m$ and $\eta(x_i) \neq 0$, for each
$1 \leq i \leq j$.
\end{prop}

\begin{prop}[{\cite[Corollary 2.8]{H2}}]
\label[prop]{Prop:notpower}
Let $R$ be a finitely generated
ring in a number field $k$.
Let $\lambda$ and
$\omega$ be non-zero elements of $R$ such that
$\omega$ is not a multiplicative
power of $\lambda$. Then there exist a finite
ring $S$ and a 
homomorphism
$\eta \from R \rightarrow S$ such that
$\eta(\omega)$ is not a multiplicative power
of $\eta(\lambda)$.
\end{prop}

We will prove a result (\Cref{Prop:Algebraic}) that strengthens and combines \Cref{Prop:orderm} and \Cref{Prop:notpower}.  The proof of this result is modeled on the proofs
of  \cite[Proposition 5]{H1} and  \cite[Theorem 2.7]{H2}.  We will need standard terminology and results of algebraic number theory.  For
references, see Janusz \cite{Janusz} or Cassels and Fr{\"o}hlich \cite{cassels1967algebraic}. Here, we summarize some of the information and set some notation for later use:
\smallskip

\begin{enumerate}[\:\:$(1)$]
\item For a number field $K$, let $\oK$ denote the ring of integers of $K$.
\item If $\pp$ is a non-zero prime ideal of $\oK$, we let $K_{\pp}$ denote  the $\pp$-adic completion of $K$ and $\oKp$ 
the ring of integers of $K_{\pp}$.   The ring $\oKp$ has a unique maximal ideal.  
The quotient of $\oKp$ by this maximal ideal is a finite field called the \emph{residue class field} of $\oKp$.  The quotient map
is called the \emph{residue class field map} with respect to $\pp$. If $R$ is a finitely generated ring in a number field $K$, then $R \subset \oKp$ for all but (at most) finitely
many primes $\pp$ of $\oK$.  By restricting the residue class field map to $R$, we obtain a finite quotient of $R$.
\item Given a non-zero prime ideal $\pp$ of $\oK$, let
$\nu_{\pp}$ denote the exponential $\pp$-adic
valuation of $K$.  We say $\pp$ \emph{divides} a non-zero element $x \in K$
if $\nu_{\pp}(x) \neq 0$.   
\item We have:
\begin{enumerate}[(i)]
\item $x \in \oKp$ if and only if $\nu_{\pp}(x) \geq 0$,
\item $x$ is contained in the unique maximal ideal
of $\oKp$ if and only if $\nu_{\pp}(x) > 0$, and
\item $x$ is a unit in $\oKp$ if and only if $\nu_{\pp}(x)  = 0$.
\end{enumerate}
\item  For any field $K$, we denote the group of $n$-th powers of
non-zero elements of $K$ by $K^n$.
\item  Let $K$ be a field, $q$ be a prime,
and $0 \neq \alpha \in K$. Then
the polynomial $x^q - \alpha$ is irreducible in $K[x]$
if and only if $\alpha \notin K^q$.
\item  We denote a primitive $r$-th root of unity
by $\zeta_r$.
\item Given a number field $K$ and a 
set of (rational) primes $P = \{ p_1, \ldots, p_s \}$,
let $K(\mu_{p_1^{\infty}}, \ldots , 
\mu_{p_s^{\infty}})$ denote the extension of
$K$ obtained by adjoining $\zeta_{p^n}$ to $K$ for
all $n \in \mathbb N$ and $p \in P$.
\item  Let $K$ be a number field and $q$ be a prime.
Suppose that $\zeta_4 \in K$ if $q = 2$, and
$\zeta_q \in K$ if $q$ is an odd prime.  Then
for any positive integer $t$, the field
$K(\zeta_{q^t})$ is a cyclic extension of
degree $q^r$ over $K$, for some $0 \leq r \leq t-1$.  
\item  Let $m = p_1^{d_1} \ \ldots \
p_s^{d_s} \in \mathbb N$ and 
$\{ \zeta_4, \zeta_{p_1}, \ \ldots\  ,  \zeta_{p_s} \} \subset K$.
Since $K(\zeta_m) = K(\zeta_{p_1^{d_1}}) \ \ldots \ 
K(\zeta_{p_s^{d_s}}),$  fact (9) implies 
that $K(\zeta_m)/K$ is cyclic of degree dividing $m$.
\end{enumerate}

\medskip

The following proposition helps describe the behavior of finitely generated multiplicative subgroups of a number field, especially regarding their images in finite quotients.  

\begin{prop}
\label[prop]{Prop:Algebraic}
Let $m \in \mathbb{N}$. Let $R$ be a finitely generated ring in a number field $k$.  
Let $\lambda$ be a non-zero element of $R$ that is not a root of unity. Let $\{ \omega_i \mid i \in I \}$ 
be a finite collection of non-zero elements in $R$ such that $\omega_i$ is not a 
multiplicative power of $\lambda$, for each $i \in I$.
Let $G_i = \langle \lambda, \omega_i \rangle$ be the (multiplicative) subgroup
of $k^{\ast} = k \setminus \{ 0 \}$ generated by $\lambda$ and $\omega_i$. 
Suppose that one of the following conditions is satisfied. 

\begin{enumerate}[\:\:$(1)$]
\item\label{Itm:Z+Z} For each $i \in I$, we have $G_i \cong \mathbb{Z} \oplus \mathbb{Z}$.
\item\label{Itm:NonDivides} For each $i \in I$, we have $\omega_i^{s_i} = \lambda^{t_i}$, for some $s_i \in \mathbb{N}$ and  $t_i \in \mathbb{Z}$
such that $s_i \nmid t_i$
\item\label{Itm:Divides} For each $i \in I$, we have $\omega_i^{s_i} = \lambda^{t_i}$, for some $s_i \in \mathbb{N}$ and  $t_i \in \mathbb{Z}$
such that $s_i \mid t_i$.
\end{enumerate}
Then there exist a finite field $F$ and a ring homomorphism $\eta\from R \rightarrow F$ such that 
$\eta(\omega_i)$ is not a multiplicative power of $\eta(\lambda)$, for each $i \in I$.

Moreover, if \eqref{Itm:Z+Z} or \eqref{Itm:NonDivides} holds, we may choose $\eta$ such that the multiplicative order of $\eta(\lambda)$ is divisible by $m$. 
If \eqref{Itm:Divides} holds, we may choose $\eta$ such that the multiplicative order of $\eta(\lambda)$ is equal to $p$, where $p$ is a sufficiently large prime. 
\end{prop}

\begin{proof}
\noindent
\emph{Case \ref{Itm:Z+Z}:} For each $i \in I$, $G_i \cong \mathbb{Z} \oplus \mathbb{Z}$. 

\medskip

Let $P = \{ p_1, p_2, \ \ldots \ , p_s \}$ be the set of prime divisors of $m$. 
By expanding $k$, if necessary, we may assume that $\{ \zeta_4, \zeta_p \mid p \in P \} \subset k$. 
Note that the infinite multiplicative order of $\omega_i$ assumed in this case implies that $\omega_i$ is not a root of unity. Therefore, by a standard argument, $\omega_i \in
k^q$ for at most finitely many primes $q \in \mathbb{Z}$. (See  \cite[proof of Case 2 of Theorem 2.7]{H2} for details.)
Let $S$ be the finite set of prime ideals of $\ok$
dividing an element of $T =  \{ \lambda, \omega_i \mid i \in I \}$, 
together with the infinite primes of $k$.  
If $\pp \notin S$, the elements of $T$ are units in $\okp$.
Let $$U_S = \{ x \in k^{\ast} \mid \nu_{\pp}(x) = 0, \
\forall \pp \notin S \}$$
denote the set of $S$-units of $k^{\ast}$.
By the $S$-Unit Theorem \cite[Theorem 8.2]{Janusz},
$U_S$ is a finitely generated
abelian group.  Let $o_i$ be the order of the torsion
subgroup of $U_S / G_i$.  Choose an odd prime
$q >  m \in \mathbb{Z}$ such that $q \nmid o_i$ and
$\omega_i \notin k^q$, for all $i \in I$.
By  \cite[Lemma 1]{H1}, there exists an element 
$\alpha  \in k(\zeta_q)(\mu_{p_1^{\infty}},\ldots, 
\mu_{p_s^{\infty}})$ such that:
(\emph{i}) $\alpha = \lambda^r$, where either
$r = 1$ or $r$ is a product of primes in $P$, and
(\emph{ii}) $\alpha \notin (k(\zeta_q)(\mu_{p_1^{\infty}},\ldots, 
\mu_{p_s^{\infty}}))^p$, for any $p \in P$.
Set $u = mr$ and let $K = k(\alpha, \zeta_q, \zeta_u)$. 
We will build two field extensions of $K$, the first to guarantee that the image of $\omega_i$ is not a multiplicative power of the image of $\lambda$, and
the second to control the multiplicative order of the image of $\lambda$. We will then form the compositum of the two extensions and apply the Tchebotarev Density 
Theorem to find a residue map with the desired properties. 

We begin with the first field extension. 
Let $d$ be the degree of $K$ over $k$ and let
$N \from K \rightarrow k$ denote the norm map.
Since $k \subset k(\zeta_q) \subset K \subset k(\zeta_q)(\mu_{p_1^{\infty}}, \cdots, \mu_{p_s^{\infty}})$ and
$\{ \zeta_{p_1}, \zeta_{p_2}, \ \ldots \ , \zeta_{p_s} \} \subset k$, the degree $d$ is a product of primes in $P$ times
the degree of $k(\zeta_q)$ over $k$. Since $[k(\zeta_q): k] < q$ and $q \notin P$, we have $q \nmid d$. 
By construction, $\omega_i \notin k^q$. 
We show that $\omega_i $ remains a non-$q$-th power in $K$.
Suppose, to the contrary, that $\omega_i = z^q$ for some $z \in K$.
Then $\omega_i^d = N(\omega_i) = N(z)^q$. Since gcd$(q, d) = 1$, there exist integers $x$ and $y$ such that 
$dx + qy = 1$. This implies that $\omega_i = \omega_i^{dx + qy} = \omega_i^{dx} \omega_i^{qy} = N(z)^{qx} \omega_i^{qy} \in k^q$,
a contradiction. 
Let $L_i$ be the splitting field of $x^q - \omega_i$ 
over $K$.  Since $\omega_i \notin K^q$, the polynomial
$x^q - \omega_i$ is irreducible over $K$.
Hence $L_i$ is a cyclic extension of $K$ of degree $q$.
Let $L_{\lambda}$ be the splitting field of $x^q - \lambda$ over $K$.
Then either $L_{\lambda} = K$, or $L_{\lambda}$ is a cyclic extension of $K$ of degree $q$.
Suppose that $L_{\lambda} = L_i$ for some $i \in I$. Then by
the classification of Kummer $q$-extensions  \cite[Chapter 3, Lemma 3]{cassels1967algebraic}, 
there exist an element $v \in K$ and an integer $c$,
relatively prime to $q$, such that
$\omega_i = \lambda^c v^q$.
Thus $\omega_i^d = N(\omega_i) =
N(\lambda^c v^q) = \lambda^{cd}
N(v)^q$, and so $N(v)^q\in G_i$.
If $N(v) \in G_i$, then $N(v)
= \lambda^a \omega_i^b$, for some $a, b \in \mathbb{Z}$.  Hence
$\omega_i^{d - bq} =
\lambda^{cd + aq}$.  Since $G_i \cong \mathbb{Z} \oplus \mathbb{Z}$, the elements $\lambda$ and $\omega_i$
are multiplicatively independent. 
It follows that $\omega_i^{d - bq} = 1$.  Therefore, $d = bq$, a contradiction.
We conclude that $N(v)$ is an element of order $q$ in $U_S / G_i$.
But this contradicts the fact that $q \nmid o_i$.  Therefore,
$L_{\lambda} \neq L_i$, for each $i \in I$. Since both extensions have prime degree $q$, it follows that 
$L_{\lambda} \cap L_i = K$.  We now combine the Kummer extensions constructed above into a single Galois extension.
The fields $L_i$  need not be distinct. After reindexing, 
we may assume that $\{ L_1, L_2, \ \ldots \ , L_n \}$ are the distinct members of 
the collection $\{ L_i \mid i \in I \}$.
Then the compositum $L = L_{\lambda} L_1 L_2 \ \ldots  \ L_n$ is a 
Galois extension of $K$ of degree $q^n$ or $q^{n + 1}$, and the map 
\begin{align*}
\mathrm{Gal}(L/K) &\rightarrow
\mathrm{Gal}(L_{\lambda}/K) \times \mathrm{Gal}(L_1/K) \times \ \ldots \ \times \mathrm{Gal}(L_n/K) \\
\gamma & \mapsto (\gamma |_{L_{\lambda}}, \gamma |_{L_1}, \ \ldots \  , \gamma |_{L_n})
\end{align*}
 is an isomorphism.
Let $\sigma_{\lambda}$ be the identity element of Gal$(L_{\lambda}/K)$, let
$\sigma_i$ be a generator of Gal$(L_i/K)$,
and define $\sigma = (\sigma_{\lambda}, \sigma_1, \sigma_2, \ \ldots \  , \sigma_n) \in \mathrm{Gal}(L/K)$.

We build a second field extension of $K$ to control the multiplicative order of the image of $\lambda$.
By construction, $\alpha \notin K^p$, $\forall p \in P$. Therefore,
the polynomials $\{ x^p - \alpha \mid p \in P \}$ are irreducible in $K[x]$. 
Let $M$ be the compositum of the splitting fields of $\{ x^p - \alpha \mid p \in P \}$ over $K$. 
Since $\{ \zeta_p \mid p \in P \} \subset K$, $M$ is a cyclic extension of $K$ of degree $p_1 p_2 \ \ldots \ p_s$.
Let $\tau$ be a generator of Gal$(M/K)$.

Since $[L:K] = q^n$ or $q^{n + 1}$, $[M:K] = p_1 p_2 \ \ldots \ p_s$, and $q \notin P$, the degrees
$[L:K]$ and $[M:K]$ are relatively prime.  Therefore, $L \cap M = K$.  It follows that the compositum $LM$ is a Galois extension of $K$,
and the map 
\begin{align*}
\mathrm{Gal}(LM/K) &\rightarrow \mathrm{Gal}(L/K) \times \mathrm{Gal}(M/K) \\
\gamma &\mapsto (\gamma |_L, \gamma |_M) 
\end{align*}
is an isomorphism.
Consider the element $$\rho = (\sigma, \tau) = ((\sigma_{\lambda}, \sigma_1, \sigma_2, \ \ldots, \ \sigma_n), \tau) \in \mathrm{Gal}(LM/K).$$
By the Tchebotarev Density Theorem, there
exist infinitely many primes $\pp$ of $K$ with
unramified extension $\mathfrak P$ in $LM$ whose Frobenius
automorphism is $\rho$.  Choose such a prime $\pp$ satisfying
(\emph{i}) $R \subset \mathcal{O}_{K_{\pp}}$, (\emph{ii}) $\ok \cap \pp$ lies outside of $S$, and
(\emph{iii}) the characteristic of the residue class field of
$\mathcal{O}_{K_{\pp}}$ is relatively prime to $qu$.
Since $\rho$ is the Frobenius automorphism of $LM/K$
with respect to $\mathfrak P / \mathfrak p$,
Gal$((LM)_{\mathfrak P} / K_{\mathfrak p})
= \langle \rho^{\prime} \rangle$ where
$\rho^{\prime} = \rho$ on $LM$.
Let $F$ denote the residue class field of
$\mathcal{O}_{K_{\pp}}$,  and let
$\eta \from R \hookrightarrow \mathcal{O}_{K_{\pp}} \rightarrow F$
denote the residue class field map restricted to $R$.

By our choice of $\pp$, the elements $\lambda$, $\omega_i$ are units
in $\mathcal{O}_{K_{\pp}}$. Therefore, $\eta(\lambda) \neq 0$ and $\eta(\omega_i) \neq 0$.
We claim that $\eta(\omega_i)$ is not a multiplicative power of $\eta(\lambda)$.  The key point is that
$\eta(\omega_i) \notin F^q$, whereas $\eta(\lambda) \in F^q$.
To see this, let $\omega_i^{1/q}$ be a root of $x^q - \omega_i$.
Since $\omega_i^{1/q} \in L_i \setminus K$, we have
$\rho^{\prime}(\omega_i^{1/q}) =
\rho(\omega_i^{1/q}) = \sigma_i(\omega_i^{1/q}) \neq \omega_i^{1/q}$.
Therefore, $\omega_i^{1/q} \notin K_{\mathfrak p}$.
As $q$ is prime and $K_{\mathfrak p}$ contains the $q$-th roots of unity, it follows that $x^q - \omega_i$
is irreducible over $K_{\mathfrak p}$. 
Since $\eta(\omega_i) \neq 0$ and the characteristic of $F$ is not equal to $q$,
the polynomial $x^q - \eta(\omega_i)$ has no multiple roots
in $F$.  Hensel's lemma \cite[Proposition 3.5]{Janusz} 
then implies that $x^q - \eta(\omega_i)$ is irreducible over $F$,
so $\eta(\omega_i) \notin F^q$.  We next show that $\eta(\lambda) \in F^q$.
Since $\rho^{\prime}(\lambda^{1/q}) =
\rho(\lambda^{1/q}) = \sigma_{\lambda}(\lambda^{1/q})
= \lambda^{1/q}$, we have $\lambda^{1/q} \in K_{\pp}$.
Since $\lambda \in \mathcal{O}_{K_{\pp}}$, the element $\lambda^{1/q} $ is integral over
$\mathcal{O}_{K_{\pp}}$, and hence belongs to $\mathcal{O}_{K_{\pp}}$.
Thus $\lambda \in {\mathcal{O}_{K_{\pp}}}^q$, and so 
$\eta(\lambda) \in F^q$.   Every multiplicative power of $\eta(\lambda)$ 
lies in $F^q$, while $\eta(\omega_i) \notin F^q$. We conclude that $\eta(\omega_i)$
is not a multiplicative power of $\eta(\lambda)$.

It remains to show that the order of $\eta(\lambda)$ is divisible by $m$. 
Let $\alpha^{1/p}$ be a root of $x^p - \alpha$, for some $p \in P$. 
Since $\alpha^{1/p} \notin K$ and Gal($M/K) \cong
\langle \tau \rangle,$ we have $\tau(\alpha^{1/p})
\neq \alpha^{1/p}$. Therefore, $\rho^{\prime}(\alpha^{1/p}) = 
\rho(\alpha^{1/p}) = \tau(\alpha^{1/p}) \neq \alpha^{1/p}$, so $\alpha^{1/p} \notin  K_{\mathfrak p}$.  
As above, we conclude that the polynomials $\{ x^p - \eta(\alpha) \mid p \in P \}$
are irreducible over $F$.   
Since $\zeta_u \in K \subset  K_{\pp}$, the polynomial
$x^u - 1$ splits completely over $F$. As char$(F) \nmid u$,
the roots of $x^u - 1$ in $F$ are distinct.  Therefore,
$F^{\ast} = F - \{ 0 \}$ contains an element of
order $u$, and in particular $u \mid \vert F^{\ast} \vert$.
The irreducibility of $ x^p - \eta(\alpha)$ implies that 
$\eta(\alpha) \notin F^p$ for any prime $p$ dividing $u$.  
Since $F^{\ast}$ is cyclic of order divisible by $u$, and $\eta(\alpha) \notin F^p$ for every prime $p \mid u$,
we obtain $u  \mid o(\eta(\alpha))$. 
Since $u = mr$ and $\lambda = \alpha^r$, the divisibility $u  \mid o(\eta(\alpha))$
implies that $m \mid o(\eta(\lambda))$, as required.

\medskip

\noindent
\emph{Case \ref{Itm:NonDivides}:} For each $i \in I$, $\omega_i^{s_i} = \lambda^{t_i}$, for some $s_i \in \mathbb{N},  t_i \in \mathbb{Z}$
such that $s_i \nmid t_i$.

\medskip

Let $s$ be the product of the distinct elements in $\{ s_i \mid i \in I \}$. 
By \Cref{Prop:orderm}, there exist a finite field $F$ and a ring homomorphism $\eta \from R \rightarrow F$ such that
the multiplicative order of $\eta(\lambda)$ is divisible by $ms$. Suppose $\eta(\omega_i) = \eta(\lambda^n)$, for some $n \in \mathbb{Z}$ and $i \in I$. 
Then $\eta(\lambda^{ns_i}) = \eta(\omega_i^{s_i}) = \eta(\lambda^{t_i})$. It follows that $ns_i \equiv t_i \ (\mathrm{mod} \ s_i).$
We conclude that $s_i \mid t_i$, a contradiction.

\medskip

\noindent
\emph{Case \ref{Itm:Divides}:} For each $i \in I$, $\omega_i^{s_i} = \lambda^{t_i}$, for some $s_i \in \mathbb{N},  t_i \in \mathbb{Z}$
such that $s_i \mid t_i$.

\medskip

Let $s$ be as above. Write $t_i = s_ic_i$, with $c_i \in \mathbb{Z}$. Since $\omega_i^{s_i} = \lambda^{t_i} = \lambda^{s_ic_i}$, 
we have $\omega_i = \lambda^{c_i} \zeta_i$, where $\zeta_i$ is an $s_i$-th root of unity. 
Since $\omega_i$ is not a multiplicative power of $\lambda$, $\zeta_i \neq 1$. 
By \Cref{Prop:orderm}, given a sufficiently large prime $p > s$, there exist a finite field $F$ and a ring homomorphism $\eta \from R \rightarrow F$
such that $\eta(\zeta_i) \neq 1$
and the multiplicative order of $\eta(\lambda)$ is equal to $p$.  
If $\eta(\lambda^n) = \eta(\omega_i) = \eta(\lambda^{c_i} \zeta_i)$, for some $n \in \mathbb{Z}$ and $i \in I$, then 
$\eta(\zeta_i) \in \langle \eta(\lambda) \rangle$. This is a contradiction, since $\eta(\zeta_i)$ is a non-trivial $s_i$-th root of unity and $\langle \eta(\lambda) \rangle$ is a group of prime order $p > s \geq s_i$.
\end{proof}

\begin{rem}
The statement of   \Cref{Prop:Algebraic} can be strengthened in the following way. Let $X$ be a finite set of non-zero elements in $R$. Then all the conclusions of the Proposition can be achieved while also ensuring that $\eta(x)
\neq 0$ for every $x \in X$.

 In the proof of Case \ref{Itm:Z+Z}, 
if we include $X$ in $T$, then every element of $X$ is a unit in $\okp$ for each $\pp \notin S$.
Hence, we may choose $\eta: R \rightarrow F$ such that $\eta(x) \neq 0$ for each $x \in X$. In the proofs of Cases \ref{Itm:NonDivides} and \ref{Itm:Divides},
\Cref{Prop:orderm} allows us to choose $\eta: R \rightarrow F$ such that $\eta(x) \neq 0$ for every $x \in X$.
\end{rem}

\section{Proof of the Main Results}
\label{main}

To prove  \Cref{Theorem:ConjugacyDistinguished}
and \Cref{Theorem:Separability}, we need to separate a coset from either all conjugates of an element $\gamma$, or all conjugates of a 
maximal parabolic subgroup $K$. A natural conjugacy invariant to consider is the trace. This allows us to focus on finitely many values, 
since every element of $K$ has trace $\pm 2$. 

\begin{notation}[Standard form for matrices]\label{main_notation}
We will use the following notation throughout this section until the end of the proof of \Cref{Theorem:ConjugacyDistinguished}. We will consider $M = \H^3/\Gamma$ where $\Gamma$ is an $\SL(2,\C)$ lift of a subgroup of $\PSL(2,\C)$.   By \Cref{Thm:ThurstonLiftConjugate},
we may conjugate $\Gamma$ to lie in $\SL(2,k)$ for
some number field $k$.   Then, fixing a $\mathbb{Q}$ embedding of $k$ into $\mathbb{C}$,
we may view $k \subset \C$. 

If $H = \langle h \rangle$ be a loxodromic subgroup of $\Gamma$ and $g \in \Gamma$, we can replace 
$k$ by a finite degree extension, and then conjugate $\Gamma$ in $\SL(2, k)$
to ensure that
\[
 g =  \begin{pmatrix} a & b \\ c & d \\ \end{pmatrix}
 \quad \mathrm{and} \quad
h =  \begin{pmatrix} \lambda & 0 \\ 0 & {\lambda}^{-1} \\ \end{pmatrix}, 
\]
for some $a, b, c, d, \lambda \in k$, with $\vert \lambda \vert \neq 1$. Then the coset $gH$ becomes
$$
gH = \left\{ \begin{pmatrix} a & b \\ c & d \\ \end{pmatrix}  \begin{pmatrix} {\lambda}^n & 0 \\ 0 & {\lambda}^{-n} \\ \end{pmatrix}
\  \Big\vert \  n \in \Z \right\}  \: = \:  
\left \{  \begin{pmatrix} a {\lambda}^n & b {\lambda}^{-n} \\ c {\lambda}^n & d {\lambda}^{-n}\\ \end{pmatrix}  \  \Big\vert \  n \in \mathbb{Z} \right\}.
$$

If $K$ is a maximal parabolic subgroup of $\Gamma$ with generators $h_1, h_2  \in K$, we may likewise assume after conjugation that
\[
 g = \begin{pmatrix} a & b \\ c & d \\ \end{pmatrix}  , \quad 
h_1 = \pm \begin{pmatrix} 1 & 1 \\ 0 & 1 \\ \end{pmatrix}, \quad \mathrm{and} \quad
h_2 =  \pm \begin{pmatrix} 1 & \omega \\ 0 & 1 \\ \end{pmatrix},
\]
for some $a, b, c, d, \omega \in k$. Then the coset $gK$ becomes
$$
gK \subset \left \{ \pm \begin{pmatrix} a & ax + b \\ c & cx + d \\ \end{pmatrix}  \  \Big\vert \  x \in Z_\omega= \Z\oplus \omega \Z \right\}.
$$
Note, that the choice of sign on $\begin{pmatrix} a & ax + b \\ c & cx + d \\ \end{pmatrix} $ depends on the lift of $\Gamma$ to $\SL(2,k)$. 

We emphasize that the above standard forms can be chosen for either a loxodromic subgroup $H$ or a maximal parabolic subgroup $K$, but not both simultaneously.
\end{notation}

While not deep, the next lemma helps organize the proofs of this section, as we will analyze the cases when there are zero, one, or two collisions between traces of the coset $gH$ and a given trace.  

\begin{lem}\label[lem]{prop:max_collisions}
Let $z \in \C$. {Let $M = {\H}^3 / \Gamma$ be a hyperbolic $3$-manifold of finite volume, let $g \in \Gamma$, and let $H, K$ be abelian subgroups of $\Gamma$.}
\begin{enumerate}[\:\:$(1)$]
\item If $H$ is a loxodromic subgroup, there are at most two elements of $gH$ with trace $z$. 
\item If $K$ is a maximal parabolic subgroup with $g \not\in K$, there are at most two elements of $gK$ with trace $z$.
\end{enumerate}
\end{lem}

\begin{proof}
We fix the matrix forms of $g$ and either $H$ or $K$ as in Notation~\ref{main_notation}.

If $H$ is loxodromic, we note that $\trace(gh^n) = z$ if and only if $a {\lambda}^n + d {\lambda}^{-n} = z$.  This is true if and only if ${\lambda}^n$ is a root of 
$f(x) = ax^2 - z x + d$.  As $f(x)$ is quadratic, there are at most two roots. Furthermore, $\lambda^n \neq \lambda^m$ for $m \neq n$, because $| \lambda | \neq 1$.

If $K$ is maximal parabolic and $g \notin K$, we have $c \neq 0$.
Therefore, if $g h_1^m h_2^n \in gK$  has trace $ z$, then we have
$\pm(a + d + c(m + n \omega)) = z$. This gives two possible solutions for $x=m + n \omega$. Since $1$ and $\omega$ are linearly independent, $x$ uniquely determines $m$ and $n$.
\end{proof}

We next establish a proposition that assumes there are no elements in $gH$ with a given trace. 

\begin{prop}
\label[prop]{Prop:trace lemma}
Let $M = {\H}^3 / \Gamma$ be a hyperbolic $3$-manifold of finite volume. 
Let $H = \langle h \rangle$ be a loxodromic subgroup of $\Gamma$, 
let $g \in \Gamma$, and let $\traceSym$ be a non-zero algebraic number. 
If there are no elements in $gH$ with trace $\traceSym$, then there exist a finite group $G$ and a homomorphism $\varphi \from \Gamma \rightarrow G$ such that
if $\gamma \in \Gamma$ has $\trace(\gamma) = \traceSym$, then
$\varphi(\gamma)$ is not conjugate into $\varphi(gH)$. 
\end{prop}

\begin{proof}
We continue to use Notation \ref{main_notation}. As in \Cref{prop:max_collisions}, let $f(x)= ax^2 - \traceSym x + d$,
and observe that $\trace(gh^n)=\traceSym$ if and only if $\lambda^n$ is a root of $f$.
Let $\omega_1$ and $\omega_2$ be the nonzero roots of $f$. (Since $\tau \neq 0$, there is at least one nonzero root. If $0$ is a root, we have $\omega_1 = \omega_2$.)
 The consequence of this setup is that $\trace(gh^n)=\traceSym$ if and only if $\lambda^n=\omega_i$ for some $i$.

Let $R \subset k$ be the ring generated by the coefficients of the generators of $\Gamma$, as well as $\omega_1$ and $\omega_2$. 
Then $f(x)$ splits completely over $R$.

We will provide different arguments, depending on the relationship between $\lambda$, $\omega_1$, and $\omega_2$. For $i \in \{ 1, 2 \}$, 
let $G_i = \langle \lambda, \omega_i \rangle$ be the subgroup of $k^{\ast} = k \setminus \{ 0 \}$ generated by $\lambda$ and $\omega_i$. 
Since $\lambda$ is not a root of unity, $G_i$ is an infinite abelian group generated by two elements. Hence, $G_i$ is isomorphic to 
$\mathbb{Z} \oplus \mathbb{Z}$, $\mathbb{Z} \oplus \mathbb{Z}/m \mathbb{Z}$, or $\mathbb{Z}$, for some natural number $m > 1$. 
In the latter two cases,  there exist integers $s_i$ and $t_i$, not both zero, such that $\omega_i^{s_i} = \lambda^{t_i}$. Note that $s_i \neq 0$ because $\lambda$ has infinite order.
Therefore, after replacing $( s_i, t_i )$ with $(-s_i, -t_i)$, if necessary, we may assume that $s_i \in \mathbb{N}$. 
By symmetry, we can reduce to the following $6$ cases. 

\begin{enumerate}[\:\:(i)]
\item \label{case_two_ZxZ} $G_1 \cong G_2 \cong \mathbb{Z} \oplus \mathbb{Z}$
\item \label{case_neither_div} $\omega_1^{s_1} = \lambda^{t_1}$ and $\omega_2^{s_2} = \lambda^{t_2}$, for some $s_1, s_2 \in \mathbb{N}, \ t_1, t_2 \in \mathbb{Z}$ 
such that $s_1 \nmid t_1$ and $s_2 \nmid t_2$
\item \label{case_both_div}$\omega_1^{s_1} = \lambda^{t_1}$ and $\omega_2^{s_2} = \lambda^{t_2}$, for some $s_1, s_2 \in \mathbb{N}, \ t_1, t_2 \in \mathbb{Z}$ 
such that $s_1 \mid t_1$ and $s_2 \mid t_2$
\item \label{case_mixed_div} $\omega_1^{s_1} = \lambda^{t_1}$ and $\omega_2^{s_2} = \lambda^{t_2}$, for some $s_1, s_2 \in \mathbb{N}, \ t_1, t_2 \in \mathbb{Z}$ 
such that $s_1 \nmid t_1$ and $s_2 \mid t_2$

\item \label{case_one_ZxZ_no_div} $G_1 \cong \mathbb{Z} \oplus \mathbb{Z}$ and $\omega_2^{s_2} = \lambda^{t_2}$, for some $s_2 \in \mathbb{N}, \  t_2 \in \mathbb{Z}$ such that $s_2 \nmid t_2$
\item \label{case_one_ZxZ_div} $G_1 \cong \mathbb{Z} \oplus \mathbb{Z}$ and $\omega_2^{s_2} = \lambda^{t_2}$, for some $s_2 \in \mathbb{N},  \ t_2 \in \mathbb{Z}$ such that $s_2 \mid t_2$
\end{enumerate}
\medskip

\noindent
\emph{Cases \eqref{case_two_ZxZ}, \eqref{case_neither_div}, \eqref{case_both_div}:}
\medskip

Recall from above that $\trace(gh^n)=\traceSym$ if and only if $\lambda^n=\omega_i$. Thus, since there are no elements in $gH$ with trace $\traceSym$,  we have $\{ \omega_1, \omega_2 \} \cap \langle \lambda \rangle = \emptyset$.
By \Cref{Prop:Algebraic}, there exist a finite field $F$ and a ring homomorphism
$\eta \from R \rightarrow F$ such that $ \eta(\omega_1)$ and $\eta(\omega_2)$ are not multiplicative powers of
$\eta(\lambda)$.
The ring homomorphism
$\eta \from R \rightarrow F$ induces a group homomorphism $\varphi\from \Gamma \subset \SL(2, R) \rightarrow \SL(2, F).$
Suppose there exist $gh^n \in gH$ and $\gamma \in \Gamma$, with $\trace(\gamma) = \traceSym$, such that $\varphi(gh^n)$ is conjugate to $\varphi(\gamma)$. Then 
$\eta(a {\lambda}^n + d {\lambda}^{-n}) = \eta(\traceSym)$. Therefore, $\eta(\lambda^n)$ is a root of $\overline{f}(x) = \eta(a) x^2 - \eta(\traceSym) x + \eta(d)$ in $F[x]$.  The
polynomial $\overline{f}(x)$ inherits a factorization from $f(x)$. 
Since $F[x]$ is a unique factorization domain, this implies that $\eta(\lambda^{n}) \in \{ \eta(\omega_1), \eta(\omega_2) \}$,
a contradiction. 
\medskip

\noindent
\emph{Case \eqref{case_one_ZxZ_no_div}:}  
\medskip

By \Cref{Prop:Algebraic} applied to $G_1$ with $m=s_2$, there exist a finite field $F$ and a ring homomorphism $\eta \from R \rightarrow F$ such that
$\eta(\omega_1)$ is not a multiplicative power of $\eta(\lambda)$, and the multiplicative order of $\eta(\lambda)$
is divisible by $s_2$.  The ring homomorphism
$\eta \from R \rightarrow F$ induces a group homomorphism $\varphi\from \Gamma \subset \SL(2, R) \rightarrow \SL(2, F).$
Suppose there exist $gh^n \in gH$ and $\gamma \in \Gamma$, with $\trace(\gamma) = \traceSym$, such that $\varphi(gh^n)$ is conjugate to $\varphi(\gamma)$.
Then, by the argument above, 
$\eta(\lambda^{n}) \in \{ \eta (\omega_1), \eta(\omega_2) \}$. By construction, this implies that $\eta(\lambda^{n}) = \eta(\omega_2)$. 
Therefore, $\eta(\lambda^{ns_2}) = \eta(\omega_2^{s_2}) = \eta(\lambda^{t_2})$. Since the multiplicative order of $\eta(\lambda)$
is divisible by $s_2$, it follows that $ns_2 \equiv t_2 \ (\mathrm{mod} \ s_2).$
We conclude that $s_2 \mid t_2$, a contradiction.
\medskip

\noindent
\emph{Cases \eqref{case_mixed_div}, \eqref{case_one_ZxZ_div}:}
\medskip

Write $t_2 = s_2r$, for some $r \in \mathbb{Z}$. Since $\omega_2^{s_2} = \lambda^{t_2} = \lambda^{s_2r}$, we have
$\omega_2 = \lambda^r \zeta$, where $\zeta$ is an $s_2$-th root of unity. 
Since $gH$ contains no element of trace $\traceSym$, we have that $\lambda^n\ne \omega_i$ for any $n\in \Z$. In particular, $\lambda^{r} \neq \omega_1$ and $\zeta \neq 1$. 
We will construct two group homomorphisms.  
By \Cref{Prop:orderm}, there exist a prime $q > s_2$, a finite field $F_1$, and a ring homomorphism $\eta_1 \from R \rightarrow F_1$, 
such that $\eta_1(\zeta) \neq 1$, $\eta_1(\lambda^{r}) \neq \eta_1(\omega_1)$,
and the multiplicative order of $\eta_1(\lambda)$ is equal to $q$.  
By \Cref{Prop:Algebraic} applied to $G_1$ with $m=q$, there exist a finite field $F_2$, and a ring homomorphism $\eta_2 \from R \rightarrow F_2$ such that
$\eta_2(\omega_1)$ is not a multiplicative power of $\eta_2(\lambda)$, 
and the multiplicative order of $\eta_2(\lambda)$ is divisible by $q$.  For $i \in \{ 1, 2 \}$, the ring homomorphisms $\eta_i$ induce group homomorphisms
$\varphi_i\from \Gamma \subset \SL(2, R) \rightarrow \SL(2, F_i).$ Consider 
$$\varphi = \varphi_1 \times \varphi_2 \from \Gamma \rightarrow \SL(2, F_1) \times \SL(2, F_2).$$
Suppose there exist $gh^n \in gH$ and $\gamma \in \Gamma$, with $\trace(\gamma) = \traceSym$, such that $\varphi(gh^n)$ is conjugate to $\varphi(\gamma)$. Then 
$\varphi_1(gh^n)$ is conjugate to $\varphi_1(\gamma)$, and $\varphi_2(gh^n)$ is conjugate to $\varphi_2(\gamma)$. 
This implies that $\eta_1(\lambda^{n}) \in \{ \eta_1(\omega_1), \eta_1(\omega_2) \}$, and
$\eta_2(\lambda^{n}) \in \{ \eta_2(\omega_1), \eta_2(\omega_2) \}$. By construction, $\eta_2(\lambda^{n}) \neq \eta_2(\omega_1)$. 
If $\eta_1(\lambda^{n}) = \eta_1(\omega_2) = \eta_1(\lambda^{r} \zeta)$, then
$\eta_1(\zeta) \in \langle \eta_1(\lambda) \rangle$. This is a contradiction, since $\eta_1(\zeta)$ is a non-trivial $s_2$-th root of unity and 
$\langle \eta_1(\lambda) \rangle$ is a group of prime order $q > s_2$. We conclude that 
$$ 
\eta_1(\lambda^{n}) = \eta_1(\omega_1) \quad  \mbox{ and } \quad  \eta_2(\lambda^{n}) = \eta_2(\omega_2).$$
Since $\eta_2(\lambda^{n}) = \eta_2(\omega_2) = \eta_2(\lambda^{r} \zeta)$, we have $\eta_2(\lambda^{ns_2}) = 
\eta_2(\lambda^{rs_2})$. Therefore,
$ns_2 \equiv rs_2 \ (\mathrm{mod} \ q)$, which implies that $n \equiv r \ (\mathrm{mod} \ q)$ as $q$ is a prime greater than $s_2$.
Since $q \mid (n - r)$ and the multiplicative order of $\eta_1(\lambda)$ is equal to $q$, we conclude that 
$\eta_1(\lambda^{n-r}) = 1$. But then, $\eta_1(\lambda^{r}) = \eta_1(\lambda^{n}) = \eta_1(\omega_1)$, a contradiction. 
\end{proof}

We now apply \Cref{Prop:trace lemma} to prove that if $H$ is loxodromic, then the coset $gH$ is conjugacy distinguished. Moreover, we describe
a class of finite quotients of $\Gamma$ in the case where $gH$ contains conjugates of a specified element $\gamma$. 

\begin{prop}
\label[prop]{Prop:loxo lemma}
Let $M = {\H}^3 / \Gamma$ be a hyperbolic $3$-manifold of finite volume, such that $g$ and $\gamma$
are elements of $\Gamma$ and $H = \langle h \rangle$ is a loxodromic subgroup of $\Gamma$.
Let $S$ be the set of elements in $gH$ that are conjugate to $\gamma$, and note that  $|S| \leq 2$  by \Cref{prop:max_collisions}.
\begin{enumerate}[\:\:$(a)$]
\item \label{item:S_empty} If $S = \emptyset$, then there exist a finite group $G$ and a homomorphism $\varphi \from \Gamma \rightarrow G$
such that $\varphi(\gamma)$ is not conjugate into $\varphi(gH)$. 
\item \label{item:S_non_empty} If $S \neq \emptyset$, write
$S = \{gh^{u_1}, gh^{u_2} \}$, where $u_1$ and $u_2$ may coincide.
 For every natural number $m$,
there exist a finite group $G_m$ and a group homomorphism $\varphi_m:\Gamma \rightarrow$ $G_m$ such that, if
$\varphi_m(\gamma)$ is conjugate to $\varphi_m(gh^n)$ for some $n \in \Z$, then $n \equiv u_i \ (\mathrm{mod} \ m)$ for some $i \in \{ 1, 2 \}$.
\end{enumerate}
\end{prop}

\begin{proof} Let $\traceSym$ be the trace of $\gamma$.  Denote by $T$ be the set of elements in $gH$ of trace $\traceSym$ (so that $S \subset T$). By \Cref{prop:max_collisions}, we know that $|T| \leq 2$.
 
If $T = \emptyset$, then by \Cref{Prop:trace lemma}, 
there exist a finite group $G$ and a group homomorphism $\varphi\from \Gamma \rightarrow G$ such that
$\varphi(\gamma)$ is not conjugate into $\varphi(gH)$, so part \eqref{item:S_empty} is satisfied. 
Therefore, we may assume that $T$ is non-empty.
Write $T = \{ gh^{n_1}, gh^{n_2} \}$, where $n_1$ and $n_2$ may coincide. 

\begin{claim}\label{claim:group homomorphism}
For every natural number $m$, there exist a finite group $G_m$ and a homomorphism $\varphi_m: \Gamma \rightarrow G_m$
such that if $\varphi_m(\gamma)$ is conjugate to $\varphi_m(gh^n)$ for some $n \in \Z$, then $n \equiv n_i \ (\mathrm{mod} \ m)$ for some $i \in \{ 1, 2 \}$.
\end{claim}

\begin{proof}[Proof of Claim \ref{claim:group homomorphism}]

Recall that an element $gh^n \in gH$ has trace $\traceSym$ if and only if $ \lambda^n$ is a (nonzero) root of  $f(x) = ax^2-\tau x +d$. In particular, if $f$ has only one nonzero root, then $|T| = 1$.
Let $\omega_1$ and $\omega_2$ be the (not necessarily distinct) nonzero roots of $f(x)$, labeled as follows. 
If $n_1 = n_2$, then $\omega_1 = \lambda^{n_1} = \lambda^{n_2}$. If $n_1 \neq n_2$, then 
$\omega_1 = \lambda^{n_1}$ and $\omega_2 = \lambda^{n_2} $. The case in which $|T| = 1$ requires separate consideration,
since in this situation we no longer have any 
information relating $\lambda$ and $\omega_2$.
As in the proof of \Cref{Prop:trace lemma}, it suffices to consider the following three possibilities. 
Observe that the case $|T| = 2$ is included in \eqref{case_4_2_div} below with $s = 1$.

\begin{enumerate}[\:\:(i)]
\item \label{case_4_2_ZxZ} $\langle \lambda, \omega_2 \rangle \cong \mathbb{Z} \oplus \mathbb{Z}$

\smallskip

\item \label{case_4_2_no_div} $\omega_2^s = \lambda^t$, for some $s \in \mathbb{N}, \ t \in \mathbb{Z}$ such that $s \nmid t$

\smallskip

\item \label{case_4_2_div} $\omega_2^s = \lambda^t$, for some $s \in \mathbb{N}, \ t \in \mathbb{Z}$ such that $s \mid t$
\end{enumerate}
\medskip

\noindent
\emph{Cases \eqref{case_4_2_ZxZ}, \eqref{case_4_2_no_div}:} Note that in these cases, $\omega_2$ is not a multiplicative power of $\lambda$. 

By \Cref{Prop:Algebraic}, there exist a finite field $F$ and a ring homomorphism
$\eta\from R \rightarrow F$ such that $\eta(\omega_2)$ is not multiplicative power of
$\eta(\lambda)$ and the multiplicative order of $\eta(\lambda)$ is divisible by $m$. The ring homomorphism
$\eta$ induces a group homomorphism $\varphi_m\from \Gamma \subset \SL(2, R) \rightarrow \SL(2, F).$ If there exists 
$gh^n \in gH$ such that $\varphi_m(gh^n)$ is conjugate to $\varphi_m(\gamma)$,
then $\eta(\traceSym) =  \trace(\varphi(gh^n))$. It follows that $\eta(\lambda^n)
\in \{ \eta(\omega_1), \eta(\omega_2) \}$. By construction, this implies that $\eta(\lambda^n)
 = \eta(\omega_1) = \eta(\lambda^{n_1})$. We conclude that $n \equiv n_1$ (mod $m$).
\medskip

\noindent
\emph{Case \eqref{case_4_2_div}:} 
By assumption, there exists an integer $r$ such that $t = rs$. Since $\omega_2^s = \lambda^t = \lambda^{rs}$, 
we have $\omega_2= \lambda^{r} \zeta$, where $\zeta$ is an $s$-th root of unity. 

Suppose that $\gamma$ is not conjugate to $gh^r$. 
Since $\pi_1(M)$ is conjugacy separable (see \cite[Implication H.8]{AFW}), there exist a finite group $G$ and a group homomorphism
$\theta_1\from \Gamma \rightarrow G$ such that $\theta_1(\gamma)$ is not 
conjugate to $\theta_1(gh^r)$. Let $o$ be the order of $\theta_1(h)$ in $G$. 
By \Cref{Prop:orderm}, there exist a finite field $F$ and a ring homomorphism $\eta\from R \rightarrow F$ such that 
the multiplicative order of $\eta(\lambda)$ is divisible by $oms$.
The ring homomorphism $\eta$ induces a group homomorphism 
$\theta_2\from \Gamma \subset \SL(2,R) \rightarrow \SL(2,F)$. 
Now, define
\[
{\varphi_m = \theta_1 \times \theta_2\from \Gamma \rightarrow G \times \SL(2,F). }
\]
Suppose $\varphi_m(\gamma)$ is conjugate to $\varphi_m(gh^n)$, for some $n \in \Z$.
Since $\theta_1(\gamma)$ is not 
conjugate to $\theta_1(gh^r)$, we have $\theta_1(h^n) \neq \theta_1(h^r)$.
Therefore $n \not\equiv r$ (mod $o$). 
Since $\theta_2(\gamma)$ is conjugate to $\theta_2(gh^n)$, 
we have $\eta(\lambda^n) \in \{ \eta(\omega_1), \eta(\omega_2) \}$.
If $\eta(\lambda^n) = \eta(\omega_2) = \eta(\lambda^r \zeta)$, then $\eta(\lambda^{ns}) = \eta(\lambda^{rs})$, and so
$ns \equiv rs$ (mod $os$). This implies that $n \equiv r$ (mod $o$), a contradiction. We conclude that
$\eta(\lambda^n) = \eta(\omega_1) = \eta(\lambda^{n_1})$, which implies $n \equiv n_1$ (mod $m$).
 
Suppose that $\gamma$ is conjugate to $gh^r$. By \Cref{Prop:orderm}, there exist a finite field $F$ and a ring homomorphism
$\eta\from R \rightarrow F$ such that the multiplicative order of $\eta(\lambda)$ 
is divisible by $ms$. The ring homomorphism $\eta$
induces a group homomorphism $\varphi_m\from \Gamma \subset \SL(2,R) \rightarrow \SL(2,F)$. By assumption,
$gh^r \in S \subset T = \{ gh^{n_1}, gh^{n_2} \}$. If $gh^r = gh^{n_1}$, then $\omega_2 = 
\lambda ^r \zeta = \lambda^{n_1} \zeta$. Thus $\omega_2^s = \lambda^{n_1s} = \omega_1^s$. 
If $\varphi_m(\gamma)$ is conjugate to $\varphi_m(gh^n)$, for some $n \in \Z$, then
$\eta(\lambda^n) \in \{ \eta(\omega_1), \eta(\omega_2) \}$, and consequently $\eta(\lambda^{ns}) = \eta(\lambda^{n_1s})$.
It follows that $ns \equiv n_1 s \smod {ms}$, and so $n \equiv n_1 \smod m$. If $gh^r \neq gh^{n_1}$ then, by construction, 
$gh^r = gh^{n_2}$, $n_1 \neq n_2$, and $\omega_2 = \lambda^{n_2}$. If $\varphi_m(\gamma)$ is conjugate to 
$\varphi_m(gh^n)$, for some $n \in \Z$, then
$\eta(\lambda^n) \in \{ \eta(\omega_1), \eta(\omega_2) \} = \{ \eta(\lambda^{n_1}), \eta(\lambda^{n_2}) \}$. Therefore,
$n \equiv n_i \smod m$ for some $i \in \{ 1, 2 \}$. This completes the construction of
$\varphi_m\from \Gamma \rightarrow G_m$.
\end{proof}

We now complete the proof of the proposition.

If $S$ is empty, then $\gamma$ is not conjugate to $gh^{n_1}$ or $gh^{n_2}$ in $\Gamma$. 
Therefore, by the conjugacy separability of finite co-volume hyperbolic 3-manifold groups, 
there exist a finite group $G$ and a group homomorphism
$\psi\from \Gamma \rightarrow G$ such that $\psi(\gamma)$ is not 
conjugate to $\psi(gh^{n_1})$ or $\psi(gh^{n_2})$. Let $m$ be the order of $\psi(h)$ in $G$. 
By Claim \ref{claim:group homomorphism} 
there exist a finite group $G_m$ and a group homomorphism $\varphi_m\from \Gamma \rightarrow G_m$ such that if
$\varphi_m(\gamma)$ is conjugate to $\varphi_m(gh^n)$, for some $n \in \Z$, then 
$$n \equiv n_1 \smod m \ \ \mathrm{or} \ \ n \equiv n_2 \smod m.$$
Let $\varphi:= \psi \times \varphi_m\from \Gamma \rightarrow G \times G_m.$
Suppose that $\varphi(\gamma)$ is conjugate to $\varphi(gh^n)$ for some $n \in \Z$. Since
$\psi(\gamma)$ is not conjugate to $\psi(gh^{n_1})$ or $\psi(gh^{n_2})$, we have $\psi(h^n) \neq \psi(h^{n_1})$
and $\psi(h^n) \neq \psi(h^{n_2})$. Therefore, $$n \not\equiv n_1 \smod m \ \ \mathrm{and} \ \ n \not\equiv n_2 \smod m, 
\ \mathrm{a} \ \mathrm{contradiction}.$$

If $S$ is non-empty, we may assume that $gh^{n_1} = gh^{u_1} \in S \subset T$. In the construction of 
$\varphi_m$ in the proof of Claim \ref{claim:group homomorphism}, we have $n \equiv n_1 \smod m$ unless 
$n_1 \neq n_2$ and $\gamma$ is conjugate to 
$gh^{n_2}$. In this case $gh^{n_2} = gh^{u_2}$. Therefore, $\varphi_m\from \Gamma \rightarrow G_m$ satisfies the conclusion of the proposition.
\end{proof}

We next prove that a coset of an abelian subgroup is conjugacy distinguished. 

\begin{named}{Theorem~\ref{Theorem:ConjugacyDistinguished}}
Let $M = {\H}^3 / \Gamma$ be a hyperbolic $3$-manifold of finite volume, let $g \in \Gamma$, and
let $H$ be an abelian subgroup of $\Gamma$. Then the coset $gH$ is conjugacy distinguished.
\end{named}

\begin{proof}
If $H$ is the subgroup consisting of the identity element, then the result follows from the conjugacy separability of finite co-volume hyperbolic 3-manifold groups \cite[Implication H.8]{AFW}.
If $g \in H$, then $gH = H$ is conjugacy distinguished by Chagas and Zalesskii \cite[Theorem A]{CZ2}. Let $\gamma$ be an element of $\Gamma$ 
that is not conjugate into $gH$. If $\gamma$ is the identity element of $\Gamma$, then the result follows from 
{the separability of abelian subgroups, proved by Allman and Hamilton \cite{AllmanHamilton}.}
Therefore, we may assume that $H$ is nontrivial, $g \notin H$, and $\gamma$ is nontrivial.

\medskip

\noindent
\emph{Case 1:} $H = \langle h \rangle$ is loxodromic.

\medskip
This follows from \Cref{Prop:loxo lemma}, part (a). 
 
 \medskip
 
\noindent
\emph{Case 2:} $H$ is maximal parabolic.
 
\medskip
  
 Let $\traceSymTwo$ be the trace of $\gamma$, and let $X$ be the set of all elements in $gH$ with trace $\pm \traceSymTwo$.
By \Cref{prop:max_collisions}, the set $X$ is finite. By assumption, no element of $X$ is conjugate to $\gamma$ in $\Gamma$.
For the sake of applying \Cref{Prop:not conjugate} and Notation \ref{main_notation}, define $K = H$.
Therefore, by \Cref{Prop:not conjugate}, we can choose generators $h_1$ and $h_2$ of $K = H$, representing slopes that can be completed to tuples $\ss_1$ and 
$\ss_2$, such that $M(\ss_1)$ and $M(\ss_2)$ are hyperbolic manifolds. Thus, for $j = 1, 2$, we may write $M(\ss_j) = \H^3 / \Gamma(\ss_j)$. Let $\psi_{\ss_j} \from \Gamma \to \Gamma(\ss_j)$ be the quotient homomorphism, where we are now considering $\Gamma(\ss_j)$ as a Kleinian group. By \Cref{Prop:not conjugate}, no element in $\psi_{\ss_j}(X)$  is conjugate to $\psi_{\ss_j}(\gamma)$ in $\Gamma(\ss_j)$.

Using Notation \ref{main_notation} and \Cref{prop:max_collisions} with a parabolic subgroup $H=K$, recall that  $Z_{\omega} = \{ m + n\omega
\ \vert \ m, n \in \mathbb{Z} \}$ and 
\[
gH \subset  
\left \{ \pm \begin{pmatrix} a & ax + b \\ c & cx + d \\ \end{pmatrix}  \  \Big\vert \  x \in Z_{\omega} \right\}.
\]
Since $g \notin H$, and $H$ is a maximal parabolic subgroup of $\Gamma$, we have $c \neq 0$.
Therefore, $g h_1^m h_2^n \in gH$  has trace $\pm \traceSymTwo$ if and only if
$a + d + c(m + n \omega) = \pm \traceSymTwo$.  Solving for $x = m + n \omega \in Z_{\omega}$, let

\[
y_+ = y_{+1} =  \frac{\traceSymTwo - a - d}{c} \quad \text{and} \quad y_-  = y_{-1} = \frac{-\traceSymTwo - a - d}{c}.
\]
(We will abuse notation slightly by thinking of the subscripts as either symbols$(\pm)$ or numbers $(\pm 1)$, as convenient.)
Then the coset $gH$ contains an element of trace $\pm \traceSymTwo$ if and only if
$\{ y_+, y_- \} \cap Z_{\omega} \neq \emptyset$. 
Let $R \subset k$ be the ring generated by the coefficients of the generators of $\Gamma$. 
By expanding $R$ and $k$, if necessary, we may assume that  $c^{-1} \in R$, which implies $y_{\pm} \in R$. 

For each $i \in \{ \pm  \}$, we will construct a homomorphism $\varphi_i \from \Gamma \rightarrow G_i$, where $G_i$ is a finite group. 
Then we will define
\[
\varphi = \varphi_+ \times \varphi_- \from \Gamma \longrightarrow G = G_+ \times G_-
\]
and show that $\varphi(\gamma)$ is not conjugate into $\varphi(gH$). 

The definition of $\varphi_i \from \Gamma \to G_i$ depends on whether $y_i$ belongs to $Q_\omega =\{ m + n \omega \ \vert \ m, n \in \Q \}$.
Suppose that $y_i \in R \setminus Q_{\omega}$.  Then by \Cref{Prop:AddRingSep}, there exist a finite ring $S_i$ and a ring homomorphism
$\rho_i \from R \rightarrow S_i$ such that $\rho_i(y_i) \notin \rho_i(Z_{\omega})$.  This ring homomorphism induces a group homomorphism
$$\varphi_i \from \Gamma \subset \SL(2,R) \rightarrow G_i = \SL(2,S_i).$$  If there exists $gh_1^m h_2^n \in gH$ of trace $i(a + d + c(m + n \omega))$
such that $\varphi_i(gh_1^m h_2^n)$ is conjugate to $\varphi_i(\gamma)$, then $\rho_i(i(a + d + c(m + n \omega) ))= \rho_i(\traceSymTwo)$. This implies that 
$\rho_i(y_i) \in  \rho(Z_{\omega})$, a contradiction. We conclude that for each element $gh_1^m h_2^n \in gH$ of trace $i(a + d + c(m + n \omega))$, 
$\varphi_i(gh_1^m h_2^n)$ is not conjugate to $\varphi_i(\gamma)$. This completes the definition of $\varphi_i$ and $G_i$ in this case.

Suppose that $y_i \in Q_{\omega}$.
We begin by specifying the choice of Dehn filling quotient $\psi_{\ss_1}$ or $\psi_{\ss_2}$ as in \cite[proof of Claim 4.11]{FHH}. Just as in that argument, we write $y_i = (m_i + n_i \omega)/v_i$ in lowest terms, as in  \Cref{Lem:LinearIndep}. 
If $v_i = 1$, then we set $j = 1$ and work with the Dehn filling $\psi_{\ss_1} \from \Gamma \to \Gamma(\ss_1)$ for concreteness.
Assuming $v_i \neq 1$, we have either $v_i \nmid m_i$ or $v_i \nmid n_i$.   If $v_i \nmid m_i$, then we set $j=2$ and select  the Dehn filling $M(\ss_2)$ and the 
quotient map $\psi_{\ss_2} \from \Gamma \to \Gamma(\ss_2)$.  Then $\psi_{\ss_2}(H)$
is an infinite cyclic loxodromic subgroup of $\Gamma(\ss_2)$ generated by $\psi_{\ss_2}(h_1)$. Consequently, $\psi_{\ss_2}(h_1^m h_2^n) = \psi_{\ss_2}(h_1^m)$.
Similarly, if $v_i \nmid n_i$, then we set $j=1$ and select the Dehn filling 
$M(\ss_1)$ and the quotient map $\psi_{\ss_1} \from \Gamma \to \Gamma(\ss_1)$. This has the effect that $\psi_{\ss_1}(h_1^m h_2^n) = \psi_{\ss_1}(h_2^n)$.
Because the arguments for $m_i$ and $n_i$ are entirely parallel, and differ only by a substitution of symbols, we assume without loss of generality that $v_i \nmid n_i$. Hence $j=1$ and we have the Dehn filling quotient $\psi = \psi_{\ss_1} \from \Gamma \to \Gamma(\ss_1)$.

If $\psi(\gamma)$ is not conjugate into $\psi(gH)$, then by Case (1),
there exist a finite group $G$ and a group homomorphism $\pi \from \Gamma(\ss_1) \rightarrow$ $G$ such that 
$\pi(\psi(\gamma))$ is not conjugate into $\pi(\psi(gH))$. The composition $(\pi \circ \psi) \from \Gamma 
\rightarrow G$ then satisfies the conclusion of the theorem.  Therefore, we may assume that $\psi(\gamma)$
is conjugate to an element  of $\psi(gH)$. 
Let $\{\psi(g) \psi(h_2)^{u_1}, \psi(g) \psi(h_2)^{u_2} \}$ be the set of elements in $\psi(gH)$ that are conjugate to $\psi(\gamma)$ in $\Gamma(\ss_1)$. 
(If there is only one element, then $u_1 = u_2$.)  If $v_i = 1$, then $\trace(gh_1^{m_i} h_2^{n_i}) = \pm \traceSymTwo$, depending on
the trace of $h_1^{m_i} h_2^{n_i}$. By construction, no element of trace $\pm \traceSymTwo$ in $gH$ is sent to an 
element conjugate to $\psi(\gamma)$. Therefore if $v_i = 1$,
then $n_i \notin \{ u_1, u_2 \}$. Recall that if $v_i \neq 1$, then $v_i \nmid n_i$. We conclude that, in either case, 
$v_i u_1 - n_i \neq 0$ and $v_i u_2 - n_i \neq 0$.
By \Cref{Lem:LinearIndep}, 
there is an infinite collection $\PrimeSet$ of primes such that for each prime $p \in \PrimeSet$, there exist a finite field $F_{\pp,i}$ 
of characteristic $p$ and a ring homomorphism $\eta_{p,i}\from R \rightarrow F_{\pp,i}$, such that 
if $\eta_{p,i}(y_i) = \eta_{p,i}(m + n \omega)$ for some $m, n \in \mathbb{Z}$, then $v_i n \equiv n_i \smod p$ and $v_i m \equiv m_i
\smod p$. Choose a prime $p \in \PrimeSet$ such that $p \nmid (v_i u_1 - n_i)$ and $p \nmid (v_i u_2 - n_i)$.
The ring homomorphism $\eta_{p,i}$ induces a congruence quotient $$\nu_i \from \Gamma \subset \SL(2, R) \rightarrow \SL(2, F_{\pp,i}).$$
By \Cref{Prop:loxo lemma},
there exist a finite group $G_p$ and a group homomorphism $\varphi_p:\Gamma(\ss_1) \rightarrow$ $G_p$, such that  
if $\varphi_p(\psi(\gamma))$ is conjugate to $\varphi_p(\psi(gh_2^n))$, for some $n \in \Z$, then $$n \equiv u_1 \ (\mathrm{mod} \ p) \ \ \mathrm{or}
\ \ n \equiv u_2 \ (\mathrm{mod} \ p).$$ 
We can now define $\varphi_i$ in the case where $y_i \in Q_{\omega}$ by
$$G_i :=  \SL(2, F_{\pp,i}) \times G_p \ \mathrm{and} \  
\varphi_i := \nu_i \times (\varphi_p \circ \psi).$$

As promised, we set $$\varphi = \varphi_+ \times \varphi_- \from \Gamma \rightarrow G_+ \times G_-.$$
Suppose there exists $gh_1^m h_2^n \in gH$ such that $\varphi(\gamma)$ is conjugate to $\varphi(gh_1^m h_2^n)$. 
If the trace of $gh_1^m h_2^n$ is equal to  $(i)(a + d + c(m + n \omega))$, then we work with $\varphi_i$.
If $y_i \notin Q_{\omega}$, then we have an immediate contradiction.  
Therefore, we may assume that $y_i \in Q_{\omega}$. Since $\varphi_p(\psi(\gamma))$ is conjugate to $\varphi_p(\psi(gh_2^n))$,
$$n \equiv u_1 \ (\mathrm{mod} \ p) \ \ \mathrm{or}
\ \ n \equiv u_2 \ (\mathrm{mod} \ p).$$ However, since $\nu_i(\gamma)$ is conjugate to $\nu_i(g h_1^m h_2^n)$, we have
$\eta_{p,i}(\traceSymTwo) = \eta_{p,i}(i(a + d + c(m + n \omega)))$, and consequently
$\eta_{p,i}(y_i) = \eta_{p,i}(m + n \omega)$. Therefore, $v_i n \equiv n_i \smod p.$ We conclude that
$$v_i u_1 \equiv n_i \ (\mathrm{mod} \ p) \ \ \mathrm{or}
\ \ v_i u_2 \equiv n_i \ (\mathrm{mod} \ p).$$ a contradiction. 

\bigskip
 
\noindent
\emph{Case 3:} $H$ is parabolic and not maximal.
 
\medskip
Let $K$ be the maximal parabolic subgroup of $\Gamma$ containing $H$. By Case 2, we may assume that $\gamma$ is conjugate into $gK$.
Since we can replace $\gamma$ with any element in its conjugacy class, we may further assume that $\gamma = gh_0$ for some
$h_0 \in K \setminus H$. 

We first consider the case where $g \in K$. Then $\gamma$ and $gH$ are both contained in $K$. 
By \Cref{Prop:not conjugate}, there exists a hyperbolic Dehn filling $M(\ss)$ of $M$, with induced map $\psi_{\ss} \from \Gamma 
\rightarrow \Gamma(\ss)$, such that $\psi_{\ss}(K)$ is a loxodromic subgroup of $\Gamma(\ss)$ and  $\psi_{\ss}(h_0) \notin \psi_{\ss}(H)$. 
If $\psi_{\ss}(\gamma)$ is conjugate into $\psi_{\ss}(gH)$, then by \Cref{Conjugate abelian implies equal},
$\psi_{\ss}(\gamma) =  \psi_{\ss}(gh_0) \in \psi_{\ss}(gH)$.
Therefore, $\psi_{\ss}(h_0) \in \psi_{\ss}(H)$, a contradiction. 
We conclude that $\psi_{\ss}(\gamma)$ is not conjugate into $\psi_{\ss}(gH)$. 
Then by Case 1, there exist a finite group $G$ and a group
homomorphism $\pi \from \Gamma(\ss) \rightarrow G$, such that $\pi(\psi_{\ss}(\gamma))$ is not conjugate into $\pi(\psi_{\ss}(gH))$.  
The composition $(\pi \circ \psi_{\ss}) \from \Gamma \rightarrow G$ then satisfies the conclusion of the theorem.

We next consider the case where $g \notin K$. 
Let $\traceSymTwo$ be the trace of $\gamma$, and let $X$ be the set of all elements in $gK$ with trace $\pm \traceSymTwo$ that are not conjugate to $\gamma$.
By \Cref{prop:max_collisions},  this set is finite. By \Cref{Prop:not conjugate}, there exist generators
$h_1$ and $h_2$ of $K$, representing slopes that can be completed to tuples $\ss_1$ and 
$\ss_2$, such that $M(\ss_1)$ and $M(\ss_2)$ are hyperbolic, 
no element in $\psi_{\ss_j}(X)$  is conjugate to $\psi_{\ss_j}(\gamma)$ in $\Gamma(\ss_j)$, and
$\psi_{\ss_1}(h_0) \notin \psi_{\ss_1}(H)$.
With these fixed generators $h_1$ and $h_2$ for $K$, we adhere to Notation \ref{main_notation}. 
Express $h_0 = h_1^{m_1} h_2^{n_1}$ in terms of these generators. 
If $\psi_{\ss_1}(\gamma)$ is not conjugate into $\psi_{\ss_1}(gH)$, then we are done by Case 1.
Therefore, for the remainder of the proof we assume that $\psi_{\ss_1}(\gamma)$
is conjugate to $\psi_{\ss_1}(gh_1^{m_2}h_2^{n_2})$ for some $h_1^{m_2} h_2^{n_2} \in H$. 

Since $g \notin K$, we have $c \neq 0$.
{For simplicity of notation, we will assume that $\trace(h_0) = 2$. (The argument in the case where $\trace(h_0) = -2$ 
is entirely parallel.)}
Therefore, if $h_1^m h_2^n \in K$, with $\trace(h_1^m h_2^n) = 2$
and $\trace(gh_1^m h_2^n) = \trace(\gamma) = \traceSymTwo$, then $$a + d  + c(m + n \omega) = a + d + c(m_1 + n_1 \omega).$$
This happens if and only if  $$m + n \omega = m_1 + n_1 \omega =: y_+.$$ If $h_1^m h_2^n \in K$, with $\trace(h_1^m h_2^n) = -2$
and $\trace(gh_1^m h_2^n) = \trace(\gamma) = \traceSymTwo$, then $$-(a + d + c(m + n \omega)) = a + d + c(m_1 + n_1 \omega).$$ This 
happens if and only if $$m  + n \omega = -(2a + 2d)/c  - (m_1 + n_1 \omega) =: y_-.$$  

Let $R$ be the coefficient ring of $\Gamma$.
By expanding $R$, if necessary, we may assume that $c^{-1} \in R$. 
As in Case 2, we will construct two homomorphisms, $\varphi_+ \from \Gamma \rightarrow G_+$ and 
$\varphi_- \from \Gamma \rightarrow G_-$, corresponding to trace $\pm 2$ of elements in $H$.
The argument is completed by the following two claims. The first and easier claim applies to 
elements $gh$ with $\trace(h) = \trace(h_0)$, while the second and harder claim applies to elements $gh$ with $\trace(h) = - \trace(h_0)$.

\begin{claim}\label{claim:phiplus}
There exists a homomorphism $\varphi_+ \from \Gamma \rightarrow G_+$ such that $\varphi_+(\gamma)$ is not conjugate to $\varphi_+(gh)$ for any $h\in H$  with $\trace(h)=+2.$
\end{claim}

\begin{claim}\label{claim:phiminus}

There exists a homomorphism $\varphi_- \from \Gamma \rightarrow G_-$ such that $\varphi_-(\gamma)$ is not conjugate to $\varphi_-(gh)$ for any $h\in H$  with $\trace(h)=-2.$
\end{claim}

Assuming the claims, define $\varphi = \varphi_+ \times \varphi_- \from \Gamma \rightarrow G_+ \times G_-$. 
Suppose there exists $gh \in gH$ such that $\varphi(\gamma)$ is conjugate to $\varphi(gh)$. 
If $\trace(h) = 2$, then Claim \ref{claim:phiplus} provides a contradiction.
If $\trace(h) = -2$, then Claim \ref{claim:phiminus} provides a contradiction.
 
 This completes the proof of the theorem modulo the claims.
 \end{proof}
 
 \begin{proof}[Proof of Claim \ref{claim:phiplus}] 
The definition of $\varphi_+$ will depend on $y_+ = m_1 + n_1 \omega$. 

We work with the Dehn filling $\ss_1$. 
Since $\psi_{\ss_1}(h_0) \notin \psi_{\ss_1}(H)$, it follows that $\psi_{\ss_1}(\gamma) \notin \psi_{\ss_1}(gH)$.
Choose $r \in \N$ such that $\psi_{\ss_1}(H) = \langle (\psi_{\ss_1}(h_2))^{r} \rangle$. Since $\psi_{\ss_1}(h_0) = \psi_{\ss_1}(h_2)^{n_1} 
\not\in \psi_{\ss_1}(H)$, we have that $r \nmid n_1$. 
Recall we have assumed that $\psi_{\ss_1}(\gamma)$
is conjugate to $\psi_{\ss_1}(gh_1^{m_2}h_2^{n_2})$ for some $h_1^{m_2} h_2^{n_2} \in H$. 
Since $\psi_{\ss_1}(\gamma)  = \psi_{\ss_1}(gh_2^{n_1}) \notin \psi_{\ss_1}(gH)$ and 
$ \psi_{\ss_1}(g h_1^{m_2} h_2^{n_2}) = \psi_{\ss_1}(g h_2^{n_2}) \in \psi_{\ss_1}(gH)$, we have that $n_1 \neq n_2$. By \Cref{prop:max_collisions}, 
$\psi_{\ss_1}(gh_2^{n_1})$ and $\psi_{\ss_1}(gh_2^{n_2})$ are the unique elements in $\psi_{\ss_1}(gK)$ that are conjugate to 
$\psi_{\ss_1}(\gamma)$ in $\Gamma(\ss_1)$. 

Let $\PrimeSet$ be the infinite set of primes guaranteed by \Cref{Lem:LinearIndep} applied to $y_{\ast} = y_+$. 
For each prime $p \in \PrimeSet$, there exist a finite field $F_{\pp}$ of characteristic $p$ and a congruence quotient 
$\nu_p: \Gamma \rightarrow \SL(2, F_{\pp})$ such that if $\nu_p(\gamma)$ is conjugate to
$\nu_p(gh_1^m h_2^n)$ for some $gh_1^m h_2^n \in gK$ with $\trace(h_1^m h_2^n) = 2$, then 
$m \equiv m_1 \smod p$ and $n \equiv n_1 \smod p$. By choosing a sufficiently large $p \in \PrimeSet$, we ensure that 
 $p \nmid (n_1 - n_2)$. 
 
 By \Cref{Prop:loxo lemma},
there exist a finite group $G_{rp}$ and a group homomorphism $\varphi_{rp} \from \Gamma(\ss_1) \rightarrow$ $G_{rp}$ such that, if
$\varphi_{rp}(\psi_{\ss_1}(\gamma))$ is conjugate to $\varphi_{rp}(\psi_{\ss_1}(gh_1^m h_2^n)) = \varphi_{rp}(\psi_{\ss_1}(g h_2^n))$, then 
$$n \equiv n_1 \smod {rp} \ \ \mathrm{or} \ \ n \equiv n_2 \smod {rp}.$$
Let $$\varphi_+ := (\varphi_{rp} \circ \psi_{\ss_1}) \times \nu_p \from \Gamma \rightarrow G_{rp} \times \SL(2, F_{\pp}).$$  
Suppose that $\varphi_+(\gamma)$ is conjugate to
$\varphi_+(gh_1^m h_2^n)$, for some $g h_1^m h_2^n \in gH$ with $\trace(h_1^m h_2^n) = 2$. Since $h_1^m h_2^n \in H$, we observe $r \mid n$.
If $n \equiv n_1 \smod {rp}$, then $n \equiv n_1 \smod r$. 
We  conclude that $r \mid n_1$, a contradiction. Therefore, we consider $n\equiv n_2 \smod{rp}$, which implies that $n \equiv n_2 \smod{p}$. 
 However, since $\nu_p(\gamma)$ is conjugate to
$\nu_p(gh_1^m h_2^n)$, we have $n \equiv n_1$ (mod $p$).
Therefore, $n_1 \equiv n_2$ (mod $p$), a contradiction.
\end{proof}

\begin{proof}[Proof of Claim \ref{claim:phiminus}]
The definition of $\varphi_-$ will depend on $y_- = -(2a + 2d)/c - (m_1 + n_1 \omega)$. 

We will consider four cases. 

\medskip

\noindent
\emph{Case 1:} $y_- \in R \setminus Q_{\omega}$
 
 \medskip
 
By \Cref{Prop:AddRingSep}, there exist a finite ring $B$
and a ring homomorphism $\rho \from R \rightarrow B$ such that
$\rho(y_-) \notin \rho(Z_{\omega})$. This induces 
$\varphi_- \from \Gamma \subset \SL(2, R) \rightarrow \SL(2,B)$. Suppose that $\varphi_-(\gamma)$ is conjugate to
$\varphi_-(gh_1^m h_2^n)$, for some $g h_1^m h_2^n \in gH$ with $\trace(h_1^m h_2^n) = -2$. Then 
$\eta(y_-) =  \eta(m + n \omega)$, a contradiction. 

\medskip

Before stating the subsequent cases, we assume that $y_- \in Q_{\omega}$. Therefore, we may write
 $$y_- = {\frac{m_{\ast} +  n_{\ast} \omega}{v_{\ast}}}$$ in lowest terms, as in \Cref{Lem:LinearIndep}.
 
Let $\PrimeSet$ be the infinite set of primes guaranteed by \Cref{Lem:LinearIndep} applied to $y_{\ast} = y_-$. 
For each prime $p \in \PrimeSet$, there exist a finite field $F_{\pp}$ of characteristic $p$ and a congruence quotient 
$\nu_p: \Gamma \rightarrow \SL(2, F_{\pp})$ such that if $\nu_p(\gamma)$ is conjugate to
$\nu_p(gh_1^m h_2^n)$ for some $gh_1^m h_2^n \in gK$ with $\trace(h_1^m h_2^n) = -2$, then 
$v_* m \equiv m_* \smod p$ and $v_* n \equiv n_* \smod p$.

 \medskip
 \noindent
\emph{Case 2:} $y_- \in Q_{\omega} \setminus Z_{\omega}$ 

\medskip

By \Cref{Prop:loxo lemma}, there exist $n^{\prime}, m^{\prime} \in \Z$ 
with the following property.  For every prime $p$, there exist finite groups $G_{p,j}$ and homomorphisms $\varphi_{p,j} \from \Gamma(\ss_j)  \to G_{p,j}$
such that  if $(\varphi_{p,1} \circ \psi_{\ss_1})(\gamma)$ is conjugate to $(\varphi_{p,1} \circ \psi_{\ss_1})(gh_1^m h_2^n)$, for some $m, n \in \Z$,
 then $$n \equiv n_1 \ (\mathrm{mod} \ p) \ \ \mathrm{or}
\ \ n \equiv n^{\prime} \ (\mathrm{mod} \ p).$$  Moreover, if $(\varphi_{p,2} \circ \psi_{\ss_2})(\gamma)$ is conjugate to $(\varphi_{p,2} \circ \psi_{\ss_2})(gh_1^m h_2^n)$, 
for some $m, n \in \Z$, then $$m \equiv m_1 \ (\mathrm{mod} \ p) \ \ \mathrm{or}\ \ m \equiv m^{\prime} \ (\mathrm{mod} \ p).$$
The integers $n_1$ and $m_1$ are the exponents of $\psi_{\ss_1}(\gamma)$ and $\psi_{\ss_2}(\gamma)$, respectively.
If two distinct elements in $\psi_{\ss_j}(gK)$ are conjugate to $\psi_{\ss_j}(\gamma)$,
then $n^{\prime}$ and $m^{\prime}$ are the exponents of the second elements  in $\psi_{\ss_1}(gK)$ and $\psi_{\ss_2}(gK)$, respectively.

Since $y_-$ is in lowest terms and $v_* \ne 1$, the system of equations
 $$v_{\ast} m = m_{\ast} \ \ \ v_{\ast} n = n_{\ast}$$ has no solution $(m, n) \in \Z^2$.  In particular, $(m_1,n_1)$, $(m_1,n^{\prime})$, $(m^{\prime},n_1)$, and $(m^{\prime},n^{\prime})$ are not solutions. 
Therefore, we can choose a prime $p \in \PrimeSet$ that satisfies all of the following conditions:
\begin{eqnarray*}
 p \nmid (v_{\ast} m_1 - m_{\ast}) \ & \text{or} \quad p \nmid (v_{\ast} n_1 - n_{\ast}) \\
 p \nmid (v_{\ast} m_1 - m_{\ast}) \ & \text{or} \quad p \nmid (v_{\ast} n^{\prime} - n_{\ast}) \\
 p \nmid (v_{\ast} m^{\prime} - m_{\ast}) \ & \text{or} \quad p \nmid (v_{\ast} n_1 - n_{\ast}) \\
 p \nmid (v_{\ast} m^{\prime} - m_{\ast}) \ & \text{or} \quad p \nmid (v_{\ast} n^{\prime} - n_{\ast})
\end{eqnarray*}

Let $\varphi_- := (\varphi_{p,1} \circ \psi_{\ss_1}) \times (\varphi_{p,2} \circ \psi_{\ss_2}) \times \nu_p.$
Suppose for the sake of contradiction that $\varphi_-(\gamma)$ is conjugate to
$\varphi_-(gh_1^m h_2^n)$, for some $g h_1^m h_2^n \in gH$ with $\trace(h_1^m h_2^n) = -2$.
Since $\nu_p(\gamma)$ is conjugate to $\nu_p(gh_1^m h_2^n)$, $$v_{\ast} m \equiv m_{\ast} \smod p \ \ \mathrm{and} \ \ 
v_{\ast} n \equiv n_* \smod p.$$ Since $((\varphi_{p,1} \circ \psi_{\ss_1}) \times (\varphi_{p,2} \circ \psi_{\ss_2}))(\gamma)$
is conjugate to $((\varphi_{p,1} \circ \psi_{\ss_1}) \times (\varphi_{p,2} \circ \psi_{\ss_2}))(gh_1^m h_2^n)$,
 $$m \equiv m_1 \ \mathrm{or} \ m^{\prime}  \smod p \ \ \mathrm{and} \ \ n \equiv n_1 \ \mathrm{or}  \  n^{\prime} \smod p.$$
 This contradicts our choice of $p$.

\medskip

In the remaining cases, $y_- \in \Z_{\omega}$. Consequently, we have $v_{\ast} = 1$.

\medskip
\noindent
\emph{Case 3:} $y_- \in Z_{\omega}$, and $\gamma$ is not conjugate to $\sigma = gh_1^{m_{\ast}} h_2^{n_{\ast}}$ 

\medskip

Just as in the proof of Claim \ref{claim:phiplus}, we work with the Dehn filling $\ss_1$. 
By that proof,
we have that $n_1 \neq n_2$ and, therefore, 
$\psi_{\ss_1}(gh_2^{n_1})$ and $\psi_{\ss_1}(gh_2^{n_2})$ are the unique elements in $\psi_{\ss_1}(gK)$ that are conjugate to 
$\psi_{\ss_1}(\gamma)$ in $\Gamma(\ss_1)$. 

Note that $\trace(\sigma) = \pm \traceSymTwo$, depending on the trace of $h_1^{m_{\ast}} h_2^{n_{\ast}}$. 
Therefore, by construction, $\psi_{\ss_1}(\sigma)$ is not conjugate to $\psi_{\ss_1}(\gamma)$. It follows that $n_2 \neq n_{\ast}$.  
Choose a prime $p \in \PrimeSet$ such that $p \nmid (n_2 - n_{\ast})$.
Let $r$ be defined as in the proof of Claim \ref{claim:phiplus} and recall that $r \nmid n_1$.
By \Cref{Prop:loxo lemma}
there exist a finite group $G_{rp}$ and a group homomorphism $\varphi_{rp} \from \Gamma(\ss_1) \rightarrow$ $G_{rp}$ such that, if
$\varphi_{rp}(\psi_{\ss_1}(\gamma))$ is conjugate to $\varphi_{rp}(\psi_{\ss_1}(gh_1^m h_2^n)) = \varphi_{rp}(\psi_{\ss_1}(g h_2^n))$, then 
$$n \equiv n_1 \smod {rp} \ \ \mathrm{or} \ \ n \equiv n_2 \smod {rp}.$$
Let $$\varphi _- := (\varphi_{rp} \circ \psi_{\ss_1}) \times \nu_p \from \Gamma \rightarrow G_{rp} \times \SL(2, F_{\pp}).$$
Suppose for the sake of contradiction that  $\varphi_-(\gamma)$ is conjugate to
$\varphi_-(gh_1^m h_2^n)$, for some $g h_1^m h_2^n \in gH$ with $\trace(h_1^m h_2^n) = -2$.
If $n \equiv n_1 \smod {rp}$, then $n \equiv n_1 \smod r$. 
Since $r \mid n$, we conclude that $r \mid n_1$, a contradiction. 
If $n \equiv n_2 \smod {rp}$, then $n \equiv n_2 \smod p$. 
However, since $\nu_p(\gamma)$ is conjugate to
$\nu_p(gh_1^m h_2^n)$, we have that $n \equiv n_* \smod p$.
 Therefore, $n_2 \equiv n_{\ast}$ (mod $p$), a contradiction.
 
 \bigskip
 \noindent
 \emph{Case 4:} $y_- \in Z_{\omega}$ and $\gamma$ is conjugate to $\sigma = gh_1^{m_{\ast}} h_2^{n_{\ast}}$ 
 
 \medskip
 
In this case, since $\gamma$ is not conjugate into $gH$, we have that 
 $h_1^{m_{\ast}} h_2^{n_{\ast}} \in K \setminus H$.  By definition, $$y_- = -(2a + 2d)/c - (m_1 + n_1 \omega) = m_{\ast} + n_{\ast} \omega.$$
 If $m_1 + n_1 \omega = m_{\ast} + n_{\ast} \omega$, then $\trace(\gamma) = a + d + c(m_1 + n_1 \omega) = 0$, a contradiction.  We conclude that 
 $\gamma \neq \sigma$. By \Cref{Prop:not conjugate}, there exists a 
hyperbolic Dehn filling $M(\tt)$ of $M$, with induced homomorphism $\psi_{\tt} \from \Gamma 
\rightarrow \Gamma(\tt)$, such that $\psi_{\tt}(K)$ is a loxodromic subgroup of $\Gamma(\tt)$ and
$\psi_{\tt}(h_1^{m_{\ast}} h_2^{n_{\ast}}) \notin \psi_{\tt}(H)$. 
By choosing $\tt$ to be sufficiently long, we may assume that $\psi_{\tt}(\gamma) \neq \psi_{\tt}(\sigma)$. 
Let $\delta$ be a generator of $\psi_{\tt}(K )$, and write $\psi_{\tt}(h_1) = \delta^{-r}$ and $\psi_{\tt}(h_2) = \delta^s$,
for some $r,s \in \Z$. Then $\psi_{\tt}(\gamma) = \psi_{\tt}(g) \delta^{-m_1 r + n_1 s}$ and $\psi_{\tt}(\sigma)
= \psi_{\tt}(g) \delta^{- m_{\ast} r + n_{\ast} s}$. By our choice of $\tt$, $-m_1 r+ n_1 s \neq - m_{\ast} r + n_{\ast} s$.
Fix $p \in \PrimeSet$ such that $p \nmid ((-m_1 r + n_1 s )- (- m_{\ast} r + n_{\ast} s))$. Write
$\psi_{\tt}(H) = \langle \delta^u \rangle$, $u \in \Z$. Since $\psi_{\tt}(h_1^{m_{\ast}} h_2^{n_{\ast}}) \notin \psi_{\tt}(H)$, 
we have $u \nmid  (- m_{\ast} r + n_{\ast} s)$. 
As $\psi_{\tt}(\sigma)$ and $\psi_{\tt}(\gamma)$ are conjugate and distinct, \Cref{prop:max_collisions} implies they are the unique elements in $\psi_{\tt}(gK)$ 
that are conjugate to $\psi_{\tt}(\gamma)$ in $\Gamma(\tt)$. 
By \Cref{Prop:loxo lemma}
there exist a finite group $G_{pu}$ and a homomorphism $\varphi_{pu} \from \Gamma(\tt) \rightarrow$ $G_{pu}$ such that, if
$\varphi_{pu}(\psi_{\tt}(\gamma))$ is conjugate to $\varphi_{pu}(\psi_{\tt}(gh_1^m h_2^n)) = \varphi_{pu}(\psi_{\tt}(g) \delta^{-mr + ns})$, then 
$$-mr + ns  \equiv -m_1 r+ n_1 s \smod {pu} \ \ \mathrm{or} \  -mr + ns \equiv -m_{\ast} r + n_{\ast} s \smod {pu}.$$
Let $$\varphi_- := (\varphi_{pu} \circ \psi_{\tt}) \times \nu_p \from \Gamma \rightarrow G_{pu} \times \SL(2, F_{\pp}).$$
Suppose for the sake of contradiction $\varphi_-(\gamma)$ is conjugate to
$\varphi_-(gh_1^m h_2^n)$, for some $g h_1^m h_2^n \in gH$ with $\trace(h_1^m h_2^n) = -2$.
Since $h_1^m h_2^n \in H$, $u \mid (-m r + n s)$. If
$$-mr + ns \equiv -m_{\ast} r + n_{\ast} s \smod{pu}, \ \ \mathrm{then} \ \ -mr + ns \equiv -m_{\ast} r+ n_{\ast} s \smod u.$$
We conclude that $u \mid (- m_{\ast} r + n_{\ast} s)$, a contradiction. 
If $$-m r + n s \equiv -m_1 r + n_1 s \smod {pu}, \ \ \mathrm{then} \ \ -m r + n s \equiv -m_1 r + n_1 s \smod {p}.$$
Since $\nu_p(\gamma)$ is conjugate to $\nu_p(gh_1^m h_2^n)$, $$m \equiv m_{\ast} \smod p \ \ \mathrm{and} \ \ 
n \equiv n_* \smod p.$$ Therefore, $$- m_{\ast} r + n_{\ast} s \equiv -m_1 r + n_1 s \smod p,$$ a contradiction.

This completes the definition of $\varphi_-$ over the four subcases and concludes the proof of the claim.
\end{proof}

We conclude this section with the proof that a coset of a loxodromic subgroup is conjugacy distinguished
from the class of maximal parabolic subgroups of $\pi_1(M)$. From this point forward, we will no longer rely on the specific conjugations set by Notation \ref{main_notation}.

\begin{named}{Theorem~\ref{Theorem:Separability}}
Let $M = {\H}^3 / \Gamma$ be a hyperbolic $3$-manifold of finite volume. Let $H = \langle h \rangle$ be a 
loxodromic subgroup of $\Gamma$ and let $K$ be a maximal parabolic subgroup of $\Gamma$. 
Let $g \in \Gamma$ be an element such that $K$ is disjoint from every conjugate of $gH$. 
Then there exists a homomorphism $\varphi\from \Gamma \rightarrow G$, where $G$ is a finite group, such that
$\varphi(K)$ is disjoint from every conjugate of $\varphi(gH)$.
\end{named}

\begin{proof}  
By assumption, $g \notin H$. We will consider two cases.
\medskip

\noindent
\emph{Case 1:} The coset $gH$ does not contain a parabolic element. 

\medskip

In this case, the coset $gH$ does not contain an element of trace $\pm 2$.
By \Cref{Prop:trace lemma}, there exist a finite group $G_1$ and a group homomorphism $\varphi_1 \from \Gamma \rightarrow G_1$ such that
if $\gamma \in \Gamma$ with $\trace(\gamma) = 2$, then $\varphi_1(\gamma)$ is not conjugate into $\varphi_1(gH)$. Similarly,
there exist a finite group $G_2$ and a group homomorphism $\varphi_2 \from \Gamma \rightarrow G_2$ such that
if $\gamma \in \Gamma$ with $\trace(\gamma) = -2$, then $\varphi_2(\gamma)$ is not conjugate into $\varphi_2(gH)$.
Consider $\varphi = \varphi_1 \times \varphi_2 \from  \Gamma \rightarrow G_1 \times G_2$. 
Suppose there exists $gh^n \in gH$ such that $\varphi(gh^n)$ is conjugate into $\varphi(K)$. Then there exists an element $\gamma \in K$ such that
$\varphi_1(gh^n)$ is conjugate to $\varphi_1(\gamma)$, and $\varphi_2(gh^n)$ is conjugate to $\varphi_2(\gamma)$.  This contradicts the fact that $\trace(\gamma) \in \{ 2, -2 \}$. 

\medskip

\noindent
\emph{Case 2:}  The coset $gH$ contains a parabolic element. 

\medskip

In this case, since $K$ is maximal and $gH$ is disjoint from every conjugate of $K$, there must be at least two cusps in $M$. 
Let $A$ be the cusp of $M$ corresponding to $K$, and let $C_1, \dots, C_\ell$ be the remaining cusps. 
We will leave the cusp $A$ unfilled, and will fill the remaining cusps. 
By Thurston's hyperbolic Dehn surgery theorem, so long as a tuple of slopes 
$\ss$ on $C_1, \dots, C_\ell$ avoids finitely many slopes on each cusp, 
the Dehn filled manifold $M(\ss)$ will be hyperbolic. 

{For each such tuple $\ss$, the fundamental group $\pi_1(M(\ss))$ has a discrete, faithful representation to a group of isometries 
$\Gamma(\ss) \subset \PSL(2,\C)$. By \Cref{Thm:ThurstonLiftConjugate}, we view $\Gamma({\bf s})$ as a subgroup of $\SL(2, \mathbb{C})$.}
By choosing a sufficiently long Dehn filling $\ss$ with associated homomorphism $\psi_\ss  \from \Gamma \rightarrow \Gamma(\ss)$, we will ensure that:
\begin{enumerate}[\:\:(1)]
\item $\psi_\ss(H)$ is a loxodromic subgroup; 
\item {$\psi_\ss$ is injective on the maximal loxodromic subgroup $H'$ containing $H$;}
\item $\psi_\ss( g) \not \in \psi_\ss(H)$;  and
\item $\psi_\ss(gH)$ does not contain 
any parabolic elements.
\end{enumerate}

The first two conditions can be checked together. Let $h'$ be a generator for $H'$.
For any sufficiently long Dehn filling, $\trace(\psi_{\ss}(h'))$ is arbitrarily close to $\trace(h')$. This ensures that $\psi_\ss(h')$ is loxodromic, and therefore that $\psi_\ss$ is injective on $H' = \langle h' \rangle$.

Now, we check the third condition.
If $g \not\in H^\prime$, then $[g,h]$ is non-trivial in $\Gamma$. Thus, for any sufficiently long Dehn filling slope, $\psi_{\ss}([g,h])$ is non-trivial, which implies that $\psi_\ss( g) \not \in \psi_\ss(H)$.
If $g \in H^\prime$, then $gH \subset H^\prime$. 
Since $\psi_\ss$ is injective on $H'$, this implies that $\psi_\ss( g) \not \in \psi_\ss(H)$.

 To see the fourth condition, note that the image of any parabolic element in $gH$ will be loxodromic, since the only parabolic
elements of $\Gamma$ that remain parabolic after the Dehn filling lie in conjugates of $K$, and the coset $gH$ is disjoint from every conjugate of $K$. 
Let $S$ be the set of loxodromic elements of $gH$.  Each element of $S$ stabilizes an axis in ${\mathbb H}^3$, whose quotient is a closed geodesic in $M$. 
By   \Cref{Lem:AxesInNeighborhood}, the union of all closed geodesics corresponding to the elements of $S$ is contained in a compact subset of $M$. 
By  \Cref{Lem:RemainLoxodromic}, for any sufficiently long Dehn filling of $M$, the image of every element of $S$ remains loxodromic. 
Therefore, the image $\psi_\ss(gH)$ does not contain 
any parabolic elements, as required.

{Now, $\Gamma(\ss)$ satisfies the hypotheses of Case 1. By that case,} there exist a finite group $G$
and a homomorphism $\varphi\from \Gamma(\ss) \rightarrow G$ such that
$\varphi(\psi_\ss(gH))$ is disjoint from every conjugate of $\varphi(\psi_\ss(K))$. The composition $\varphi  \circ \psi_\ss \from \Gamma \rightarrow G$ then
satisfies the conclusion of the theorem. 
\end{proof}

\bibliographystyle{plain}
\bibliography{ConjDistinguished}

\end{document}